\documentclass{article}%
\usepackage{amsmath}
\usepackage{amsfonts}
\usepackage{amssymb}
\usepackage{graphicx}%
\newtheorem{theorem}{Theorem}

\newtheorem{corollary}[theorem]{Corollary}

\newtheorem{definition}[theorem]{Definition}
\newtheorem{example}[theorem]{Example}

\newtheorem{lemma}[theorem]{Lemma}

\newtheorem{proposition}[theorem]{Proposition}
\newtheorem{remark}[theorem]{Remark}

\newenvironment{proof}[1][Proof]{\noindent\textbf{#1.} }{\ \rule{0.5em}{0.5em}}

\begin{document}

\title{Cyclic $n$-gonal surfaces and their automorphism groups - UNED Geometry Seminar}
\author{S. Allen Broughton, Rose-Hulman Institute of Technology
\and Aaron Wootton, University of Portland}
\maketitle

\begin{abstract}
This paper is based upon two lectures on the authors' joint work, presented by
the first author at the UNED Geometry Seminar in February-March, 2009. As the
detailed statements and proofs of results presented in the talks will be
published elsewhere, this paper will only give an overview of cyclic $n$-gonal
surfaces, their automorphism groups, and some examples illustrating the
computational methods.

\end{abstract}
\tableofcontents

\section{Cyclic $n$-gonal surfaces\label{sec-intro}}

This paper is based upon two lectures \cite{Brou1}, \cite{Brou2} on the
author's joint work, presented by the first author at the UNED Geometry
Seminar in February and March of 2009. The full results for those talks will
be given in the forthcoming papers \cite{Brou3} and \cite{Brou-Woo2}, so in
this paper we content ourselves with an overview of cyclic $n$-gonal surfaces,
their automorphism groups, and a small number of examples illustrating the
computational methods.

A closed Riemann surface $S$ of genus $\sigma\geq2$ is called {cyclic
$n$-gonal} if there exists a cyclic group of automorphisms $C$ of order $n$
such that the quotient space $S/C$ has genus $0$. Such surfaces are of great
interest since they have a simple plane model given in equation \ref{eq-basic}
below. The map $\pi_{C}:S\rightarrow S/C=\mathbb{P}^{1}$, where $\mathbb{P}%
^{1}$ denotes the Riemann sphere, is called a \emph{cyclic }$n$\emph{-gonal
morphism}. When $n=2$, $S$ is a hyperelliptic surface and so any generator of
$C$ should be thought of as a generalization of a hyperelliptic involution.

We consider the general problem of determining $A=\mathrm{Aut}{(S)}$, the full
conformal automorphism group of $S.$ When $C$ is normal in $A$, such as the
hyperelliptic case or the prime order and large genus case, there are
well-known methods for determining $A$ (see Section \ref{sec-top}). The normal
case has been investigated by several authors \cite{Acc}, \cite{Kont},
\cite{Sh}, \cite{Woo1}. The non-normal case has been considered in
\cite{Woo2}, and in some detail in the forthcoming work \cite{Brou-Woo2}. In
general, a classification of surfaces and automorphism groups is hopeless. See
the paper \cite{Brou-Woo2} for some examples of why the problem is complex
when there are no restrictions on the $C$-action$.$ When some restrictions are
placed on the $C$-action or, the automorphism group $A$, then some progress
can be made. Some possible restrictions are the following.

\begin{enumerate}
\item In \cite{Woo1} the cyclic group $C$ is assumed to be of prime order.

\item In \cite{Kont} the $C$-action is fully ramified, i.e., $m_{i}=n,$
or$\ \gcd(n,p_{i})=1.$ See the next section for notation.

\item In \cite{Brou-Woo2} the cyclic group $C$ is assumed to be a weakly
malnormal subgroup of $A,$ in which case $C$ is normal if the genus is large
enough. The definition is given in Section \ref{sec-const}. This case includes
the first two cases above.

\item Another case of great interest is quasi-platonic surfaces, namely
$\pi_{A}:S\rightarrow S/A=\mathbb{P}^{1}$ is branched over three points. These
cases are interesting since the surfaces are defined over number fields and
give many computable examples of dessins d'enfants. Details will be in the
forthcoming paper \cite{Brou-Woo3}.
\end{enumerate}

\subsection{Plane models}

One reason for looking at cyclic $n$-gonal surfaces is that there is some hope
in determining the equation of the surface. There is a plane model of the
form
\begin{equation}
y^{n}=f(x)=%
{\textstyle\prod\limits_{i=1}^{r}}
\left(  x-a_{i}\right)  ^{p_{i}},\label{eq-basic}%
\end{equation}
where the $p_{i}$ and $p=p_{1}+\cdots+p_{r}=\deg(f)$ satisfy
\begin{equation}
0<p_{i}<n\label{eq-basic-1}%
\end{equation}%
\begin{equation}
n\ \mathrm{divides}\text{ }p\label{eq-basic-2}%
\end{equation}%
\begin{equation}
\gcd(n,p_{1},\ldots,p_{r})=1\label{eq-basic-3}%
\end{equation}
\bigskip The closed curve $\overline{S}$ in $\mathbb{P}^{2}$ defined by
equation \ref{eq-basic} is smooth except possibly where $y=0,$ $\infty.$ The
normalization map $\nu:S\rightarrow\overline{S}$ yields a smooth curve of
genus
\begin{equation}
\sigma=\frac{1}{2}\left(  2+(r-2)n-%
{\textstyle\sum\limits_{i=1}^{r}}
d_{i}\right)  ,\label{eq-genus}%
\end{equation}
where $d_{i}=\gcd(n,p_{i}).$ The group $C$ acts on the smooth part of
$\overline{S}$ by
\begin{equation}
(x,y)\rightarrow(x,e^{2\pi ki/n}y)\label{eq-basic-4}%
\end{equation}
and this action extends to \ $S$ via $\nu.$ The $n$-gonal morphism is the
$\nu$-lift of the map
\begin{equation}
\overline{S}\rightarrow\mathbb{P}^{1},(x,y)\rightarrow x.\label{eq-basic-5}%
\end{equation}
The $n$-gonal morphism $\pi_{C}:S\rightarrow\mathbb{P}^{1}$ is ramified only
over the finite points $\left\{  a_{1},\ldots,a_{r}\right\}  ,$ and the degree
of ramification over $a_{i}$ is
\begin{equation}
m_{i}=\frac{n}{d_{i}}=\frac{n}{\gcd(n,p_{i})}.\label{eq-basic-6}%
\end{equation}

\subsection{Overview of computing ${\text{Aut}}(S)$}

For generically chosen $a_{i}$ we usually have $C=A=\mathrm{Aut}(S).$ For
special values of the $a_{i}$ and selections of the $p_{i}$ the automorphism
group may be larger. See, for instance \cite{Kont}, \cite{Sh}, \cite{Woo1},
and \cite{Woo2}. We shall not directly work with the defining equation or
plane models in this paper, but use group theoretic methods instead.

Let $N=\mathrm{Nor}_{A}(C)$ so that $C\trianglelefteq N\leq A.$ Our method is
to lift this triple to a triple of covering Fuchsian groups $\Gamma
_{C}\trianglelefteq\Gamma_{N}\leq\Gamma_{A}$ and then to employ the group
theory to implement classification. In Section \ref{sec-triples} we describe
the lifted triples, especially canonical generators and signatures. In
Sections \ref{sec-top}, \ref{sec-bot}, we describe the inclusions $\Gamma
_{C}\vartriangleleft\Gamma_{N}$ and $\Gamma_{N}<\Gamma_{A},$ respectively,
using quite different methods in the two cases. In Section \ref{sec-top} the
group $K=\Gamma_{N}/\Gamma_{C}\backsimeq N/C$ and its action on the generating
system of $C$ will be fundamental to the discussion $\Gamma_{C}%
\vartriangleleft\Gamma_{N}$. In Section \ref{sec-bot} we use permutation group
methods on the coset space $\Gamma_{A}/\Gamma_{N}$ to describe the inclusions
$\Gamma_{N}<\Gamma_{A}.$ In particular we describe the notion of families of
triples $\Gamma_{C}\trianglelefteq\Gamma_{N}\leq\Gamma_{A}$ where $\Gamma
_{N}<\Gamma_{A}$ has a fixed coset structure but $n=\left\vert C\right\vert $
varies over a family.\newline

The general computational procedure for classification is

\begin{enumerate}
\item Specify a restriction on the triples $\Gamma_{C}\trianglelefteq
\Gamma_{N}\leq\Gamma_{A}$ by imposing a group theoretic, geometric, or
arithmetic constraint as noted in the introductory paragraphs. This is
described in Section \ref{sec-const}. This step limits the complexity of calculations.

\item Compute all possible signature pairs $\left(  \mathcal{S}(\Gamma
_{N}),\mathcal{S}(\Gamma_{A})\right)  .$ See Section \ref{sec-triples} for the
definition of signature.

\item Determine the inclusions $\Gamma_{N}<\Gamma_{A}$ determining both
exceptional cases and parametric families. This is discussed in Section
\ref{sec-bot}.

\item Determine the exact sequences $\Gamma_{C}\rightarrow\Gamma
_{N}\rightarrow K.$ This is discussed in Section \ref{sec-top}.

\item Fuse the pairs $\Gamma_{N}<\Gamma_{A}$ and $\Gamma_{C}\vartriangleleft
\Gamma_{N}$ together to compute $C\trianglelefteq N\leq A,$ or demonstrate
that no extension exists. We discuss the computationally intensive methods
very briefly in subsection \ref{subsec-monowordmap}. However, because of space
limitations we use ad-hoc methods in the examples in Section \ref{sec-samp}.
Full details of the algorithms are in \cite{Brou-Woo2} and \cite{Brou3}.

\item The surface $S$ automatically exists as a quotient $\mathbb{H}/\Pi$
where $\Pi\vartriangleleft\Gamma_{C}.$ Constructing a model as in equation
\ref{eq-basic} takes much more work, see for instance \cite{Kont}, \cite{Sh},
\cite{Woo1}, and \cite{Woo2}. We do not address construction of the plane
models in this paper.
\end{enumerate}

\begin{remark}
\label{rk-orderA}The order of A is given by%
\begin{equation}
\left\vert A\right\vert =\frac{\left\vert A\right\vert }{\left\vert
N\right\vert }\frac{\left\vert N\right\vert }{\left\vert C\right\vert
}\left\vert N\right\vert =mn\left\vert K\right\vert \label{eq-OrderofA}%
\end{equation}
where
\[
n=\left\vert C\right\vert ,m=\left\vert \Gamma_{A}/\Gamma_{N}\right\vert .
\]
If the signature of $C$ is known -- say the signatures of $N$ and $K$ are
known --then
\begin{equation}
\left\vert C\right\vert =\gcd(\mathrm{periods}\text{ }\mathrm{of}\text{
}\mathrm{C})\label{eq-OrderofC}%
\end{equation}

Given $\Gamma_{N}<\Gamma_{A}$ and $\Gamma_{C}\vartriangleleft\Gamma_{N}$ then
$C\trianglelefteq N\leq A$ is determined by finding a torsion free
$\Pi\vartriangleleft\Gamma_{C}$ such that $\Pi\vartriangleleft\Gamma_{A}.$
There are infinitely many such $\Pi$ but very few result in a cyclic $C$.
Imposing an exact sequence $\Gamma_{C}\rightarrow\Gamma_{N}\rightarrow K$ with
the additional constraints given in equations \ref{eq-OrderofA} and
\ref{eq-OrderofC} eliminates many of the non-cyclic possibilities.
\end{remark}

\section{Translation to Fuchsian group triples\label{sec-triples}}

\subsection{Basics and canonical generators}

We take much of our notation from \cite{Brou-Woo2} and \cite{Woo2}. Recall
that a compact Riemann surface $S$ of genus $\sigma\geq2$ can be realized as a
quotient of the upper half plane $\mathbb{H}/\Pi$ where $\Pi$ is a torsion
free Fuchsian group called a \emph{surface group} for $S$. Under such a
realization, a group $G$ acts as a group of conformal automorphisms on $S$ if
and only if there exists an epimorphism $\eta\colon\Gamma\rightarrow G$ with
$\mathrm{ker}{(\eta)}=\Pi$ for some Fuchsian group $\Gamma$. We call $\eta$ a
\emph{surface kernel epimorphism}, and $\Gamma$ the \emph{covering Fuchsian
group} of $G$, usually denoting it by $\Gamma_{G}.$ We identify the orbit
spaces $\mathbb{H}/\Gamma$ and $S/G.$ The quotient map $\pi_{G}\colon
S\rightarrow S/G$ is branched over the same points as $\pi_{\Gamma}%
\colon\mathbb{H}\rightarrow\mathbb{H}/\Gamma$ with the same ramification
indices. We define the signature of $\Gamma$ to be the tuple
\[
\mathcal{S}(\Gamma)=(\sigma_{\Gamma};m_{1},m_{2},\dots,m_{r}),
\]
where the quotient space $\mathbb{H}/\Gamma$ has genus $\sigma_{\Gamma}$ and
the quotient map, $\pi_{\Gamma}$ (and also $\pi_{G}$) branches over $r $
points with ramification indices $m_{i}$ for $1\leq i\leq r$. We call
$\sigma_{\Gamma}$ the \emph{orbit genus} of $\Gamma$ and the numbers
$m_{1},\dots,m_{r}$ the \emph{periods} of $\Gamma$. If $\sigma_{\Gamma}=0$ the
signature may be abbreviated to $(m_{1},m_{2},\dots,m_{r}),$ which we may also
write as $(m_{1}^{e_{1}},\dots,m_{s}^{e_{s}})$ to indicate repeated periods.
The signature of $\Gamma$ provides information regarding a presentation for
$\Gamma$, and in the special case that $\sigma_{\Gamma}=0$, we have the following.

\begin{theorem}
\label{th-sig} If $\Gamma$ is a Fuchsian group with signature $(m_{1}%
,\dots,m_{r})$ then there exist an ordered set of elliptic (finite$\ $order)
group elements $\mathcal{G}=\left\{  \gamma_{1},\dots,\gamma_{r}\right\}
\subseteq PSL(2,\mathbb{R})$, such that;
\end{theorem}

\begin{enumerate}
\item $\Gamma=\langle\gamma_{1},\dots,\gamma_{r}\rangle$.

\item Defining relations for $\Gamma$ are
\begin{equation}
\gamma_{1}^{m_{1}}=\gamma_{2}^{m_{2}}=\dots=\gamma_{r}^{m_{r}}=\prod
\limits_{i=1}^{r}\gamma_{i}=1.\label{eq-rels}%
\end{equation}

\item Each non-identity elliptic element (element of finite order) lies in a
unique conjugate of $\langle\gamma_{i}\rangle$ for suitable $i$.
\end{enumerate}

\begin{definition}
We call a set of elements of $\Gamma$ satisfying 1 and 2 of Proposition
\ref{th-sig} \emph{canonical generators} of $\Gamma$ for the signature
$(m_{1},\dots,m_{r})$.
\end{definition}

\begin{remark}
\label{rk-braid}The canonical generators are not unique and the periods of the
signature may be permuted. The permutations can be built up from simple
transpositions as follows. Set
\[
m_{i}^{\prime}=m_{i+1},m_{i+1}^{\prime}=m_{i},m_{j}^{\prime}=m_{j}\text{
\textrm{otherwise}}%
\]
and
\[
\gamma_{i}^{\prime}=\gamma_{i+1},\gamma_{i+1}^{\prime}=\gamma_{i+1}^{-1}%
\gamma_{i}\gamma_{i+1},\gamma_{j}^{\prime}=\gamma_{j}\text{
\textrm{otherwise.}}%
\]
Then the $\gamma_{i}^{\prime}$ constitute a canonical generating set for the
periods $m_{i}^{\prime}$.
\end{remark}

\begin{remark}
Let $\eta\colon\Gamma\rightarrow G$ be a surface kernel epimorphism with
signature $(m_{1},m_{2},\dots,m_{r})$. Set $g_{i}=\eta(\gamma_{i}).$ Then the
vector $(g_{1},\dots,g_{r})$ of elements satisfies
\begin{equation}
o(g_{i})=o(\gamma_{i})=m_{i},1\leq i\leq r\label{eq-gv1}%
\end{equation}%
\begin{equation}
\prod\limits_{i=1}^{r}g_{i}=1\label{eq-gv2}%
\end{equation}%
\begin{equation}
G=\left\langle g_{1},\dots,g_{r}\right\rangle .\label{eq-gv3}%
\end{equation}
Any such vector is called a generating $(m_{1},\dots,m_{r})$-vector of $G$. We
call the tuple $\mathcal{S}(\Gamma)$ the \emph{branching data} or the
\emph{signature} of the $G$-action on $S$. the definition can be extended to
the case where $\sigma_{\Gamma}>0,$ but we do not need it.
\end{remark}

\begin{remark}
Using areas of fundamental regions one can show that $A(\Gamma)=\pi
(2\sigma-2)/\left\vert G\right\vert .$ Letting $\tau$ be the orbit genus of
$\sigma(S/G)=\sigma(\mathbb{H}/\Gamma)$ then we get the Riemann-Hurwitz
formula or
\[
\frac{2\sigma-2}{\left\vert G\right\vert }=2\tau-2+%
{\displaystyle\sum\limits_{i=1}^{r}}
\left(  1-\frac{1}{m_{i}}\right)  .
\]

\end{remark}

\begin{example}
\label{ex-ngonalsignature}Suppose that $C$ is the cyclic $n$-gonal group of
the surface given by equation \ref{eq-basic}. Then the signature of the
$C$-action is $(m_{1},m_{2},\dots,m_{r})$ with the $m_{i}$ given by equation
\ref{eq-basic-6}. The Riemann-Hurwitz equation applied to the $C$-action is
then
\[
\frac{2\sigma-2}{n}=\frac{2\sigma-2}{\left\vert C\right\vert }=-2+%
{\displaystyle\sum\limits_{i=1}^{r}}
\left(  1-\frac{1}{m_{i}}\right)
\]
or
\[
\sigma=\frac{1}{2}\left(  2+\left(  r-2\right)  n-%
{\displaystyle\sum\limits_{i=1}^{r}}
d_{i}\right)  ,
\]
confirming equation \ref{eq-genus}.
\end{example}

\subsection{Induced generators}

Note that Proposition \ref{th-sig} implies that, if $\Gamma\leq\Delta$, then
any elliptic element of $\Gamma$ must be conjugate to an elliptic element of
$\Delta$. This motivates the following definition.

\begin{definition}
Suppose $\Gamma\leq\Delta$ are Fuchsian groups, $\theta\in\Gamma$ is an
elliptic element and $\theta$ is $\Delta$-conjugate to a power of $\zeta
\in\Delta,$ i.e., $\theta=x\zeta^{k}x^{-1}.$ Then we say $\theta$ is induced
by $\zeta$.
\end{definition}

We note that, by Theorem \ref{th-sig}, any elliptic generator of $\Gamma$ in a
set of canonical generators for $\Gamma$ must be conjugate to a power of a
unique elliptic generator of $\Delta$ in a set of canonical generators of
$\Delta$. To determine exactly how elliptic generators of $\Gamma$ and
$\Delta$ related, we can use the following important consequence of the main
result of \cite{Sing1}.

\begin{theorem}
\label{th-Singerman1} Suppose that $\Gamma\leq\Delta$, $\zeta\in\Delta$ is a
canonical generator of order $k$ and let $\rho\colon\Delta\rightarrow
S_{[\Delta:\Gamma]}$ denote the map induced by action of $\Delta$ on the left
cosets of $\Gamma$. Then the number of canonical generators of $\Gamma$
induced by $\zeta$ is equal to the number of cycles of $\rho(\zeta)$ of orders
less than $k$ and the order of these elements are given by $k/k_{i}$ where the
$k_{i}$ run over the lengths of the cycles of $\rho(\zeta)$. Moreover, if
$\zeta^{\prime}$ is any other canonical generator of $\Delta$, then the
canonical generators of $\Gamma$ induced by $\zeta^{\prime}$ are distinct from
the ones induced by $\zeta$.
\end{theorem}

In the special case where $\Gamma\vartriangleleft\Delta$, we have the
following result.

\begin{corollary}
Suppose that $\Gamma\vartriangleleft\Delta$, $\zeta\in\Delta$ and $\theta$ is
a canonical generator of $\Gamma$ induced by $\zeta.$ Then the order
$o(\theta)$ is the same for all canonical generators $\theta$ of $\Gamma$
induced by $\zeta$ and $\zeta$ induces $\frac{[\Delta:\Gamma]}{o(\zeta
)/o(\theta)}$ distinct canonical generators of $\Gamma.$ Moreover, if
$\zeta^{\prime}$ is any other canonical generator of $\Delta$, then the
canonical generators of $\Gamma$ induced by $\zeta^{\prime}$ are distinct from
the ones induced by $\zeta$.
\end{corollary}

\subsection{The spherical group $K$}

We fix some notation. Let $S$ denote a cyclic $n$-gonal surface of genus
$\sigma$, $\Pi$ a surface group for $S$, $C$ an $n$-gonal group for $S$ and
$\Gamma_{C}$ the covering Fuchsian group of $C$. Also, let $A$ denote the full
automorphism group of $S$, $\Gamma_{A}$ its covering Fuchsian group, $N $ the
normalizer of $C$ in $A$, $\Gamma_{N}$ its covering Fuchsian group. Next let
$K=N/C=\Gamma_{N}/\Gamma_{C}$ and let $\eta\colon\Gamma_{A}\rightarrow A$ and
$\chi\colon\Gamma_{N}\rightarrow K$ denote the canonical quotient maps. The
relations are summarized in this diagram%
\begin{equation}%
\begin{array}
[c]{ccccc}%
\Gamma_{C} & \hookrightarrow & \Gamma_{N} & \hookrightarrow & \Gamma_{A}\\
\downarrow\mathcal{\eta} &  & \downarrow\mathcal{\eta} &  & \downarrow
\mathcal{\eta}\\
C & \hookrightarrow & N & \hookrightarrow & A
\end{array}
\label{eq-cover}%
\end{equation}
and the exact sequences%
\begin{equation}%
\begin{array}
[c]{ccccc}%
\Pi & \hookrightarrow & \Gamma_{A} & \overset{\mathcal{\eta}}%
{\twoheadrightarrow} & A
\end{array}
\label{eq-exact1}%
\end{equation}%
\begin{equation}%
\begin{array}
[c]{ccccc}%
\Gamma_{C} & \hookrightarrow & \Gamma_{N} & \overset{\mathcal{\chi}%
}{\twoheadrightarrow} & K
\end{array}
\label{eq-exact2}%
\end{equation}
Notice that since the group $K=N/C$ acts on the surface $S/C=\mathbb{P}^{1}$,
it follows that $K$ is a finite subgroup of $\mathrm{PSL}{(2,\mathbb{C}),}$
acting on $\mathbb{P}^{1}$ by linear fractional transformations. All such
groups are well known as well as the properties of the corresponding quotient
maps and can be thought of as a special case of Proposition \ref{th-sig}. We summarize.

\begin{theorem}
\label{th-sphere} Suppose that $K$ is a finite subgroup of $\mathrm{PSL}%
{(2,\mathbb{C})}$. Then $K$ is conjugate to one of $C_{k}$, $D_{k}$, $A_{4}$,
$S_{4}$ or $A_{5}$ (where $C_{k}$ denotes the cyclic group of order $k$ and
$D_{k}$ the dihedral group of order $k$). The quotient map $\pi_{K}%
:\mathbb{P}^{1}\rightarrow\mathbb{P}^{1}$ branches over $s$ points with
ramification indices $m_{i}$ for $1\leq i\leq s$. The signature of such a
group is the tuple $(a_{1},\dots,a_{e})$

\begin{itemize}
\item where $e=2,3$, $a_{i}\geq2$

\item $\frac{1}{a_{1}}+\cdots+\frac{1}{a_{e}}>1$

\item $a_{1}=a_{2}$ if $e=2\ $
\end{itemize}

\noindent and any such tuple corresponds to a group. We tabulate all
signatures in Table 1. Moreover, two groups in the table are isomorphic if and
only if they have the same signature.%
\[%
\begin{tabular}
[c]{c}%
$\text{\textbf{Table 1}}$\\
$%
\begin{tabular}
[c]{||l|l||}\hline\hline
Group & Signature\\\hline
$C_{k}$ & $(k,k),k\geq2$\\\hline
$D_{k}$ & $(2,2,k),k\geq2$\\\hline
$A_{4}$ & $(2,3,3)$\\\hline
$S_{4}$ & $(2,3,4)$\\\hline
$A_{5}$ & $(2,3,5)$\\\hline\hline
\end{tabular}
\ \ \ \ $%
\end{tabular}
\ \
\]

\end{theorem}

As suggested by Theorem \ref{th-sig} the interplay among the signatures of
$\Gamma_{C},$ $\Gamma_{N},$ $\Gamma_{A}$ or the signatures of $C,$ $N,$ $A, $
and the epimorphisms $\chi$ and $\eta$ are closely related to the ramification
properties of the quotient maps $S/C\rightarrow S/N\rightarrow S/A$ or
$\mathbb{H}/\Gamma_{C}\rightarrow\mathbb{H}/\Gamma_{N}\rightarrow
\mathbb{H}/\Gamma_{A}$ among the quotient surfaces. This relationship is
explained in more detail in Sections \ref{sec-top} and \ref{sec-bot}. In
Section \ref{sec-const} we discuss how the action of $A$ on the ramification
points of $S\rightarrow S/A$ is related to induced generators and the
ramification of $\mathbb{H}/\Gamma\rightarrow\mathbb{H}/\Delta.$

\subsection{Fuchsian group invariants}

We fix some more notation. We denote the signatures of a pair $\Gamma<\Delta$
(e.g., $\Gamma_{N}<\Gamma_{A})$ by $(m_{1},m_{2},\dots,m_{s})$ and
$(n_{1},n_{2},\dots,n_{t})$ respectively. Let $\mathcal{G}_{1}=\left\{
\theta_{1},\dots,\theta_{s}\right\}  $ and $\mathcal{G}_{2}=\left\{  \zeta
_{1},\dots,\zeta_{t}\right\}  $ be sets of canonical generators of
$\Gamma<\Delta$ respectively. Important Fuchsian group invariants and
invariants of pairs may be read off from the signatures.\newline

For single groups we have.

\begin{itemize}
\item The area of a fundamental region: $A(\Gamma)=2\pi\mu(\Gamma)$ where:%
\[
\mu(\Gamma)=-2+\sum_{j=1}^{s}\left(  1-\frac{1}{m_{j}}\right)  =\left(
s-2\right)  -\sum_{j=1}^{s}\frac{1}{m_{j}}.
\]
For completeness, when the genus $\sigma=\sigma(\Gamma)$ is greater than zero%
\[
\mu(\Gamma)=2(\sigma-1)+\sum_{j=1}^{s}\left(  1-\frac{1}{m_{j}}\right)  .
\]

\item \emph{Teichm\"{u}ller dimension} $d(\Gamma)$ of $\Gamma$: the dimension
of the Teichm\"{u}ller space of Fuchsian groups with signature $\mathcal{S}%
(\Gamma)$ given by%
\[
d(\Gamma)=s-3=\left\vert \mathcal{G}_{1}\right\vert -3.
\]
For completeness, when the genus $\sigma=\sigma(\Gamma)$ is greater than zero
we have%
\[
d(\Gamma)=3(\sigma-1)+s.
\]

\end{itemize}

For pairs we combine the invariants.

\begin{itemize}
\item For a finite index pair $\Gamma<\Delta$, we have
\[
\lbrack\Delta:\Gamma]=\mu(\Gamma)/\mu(\Delta)
\]

\item For finite index pair $\Gamma\leq\Delta$, we call the quantity
\[
d(\Gamma,\Delta)=d(\Gamma)-d(\Delta)
\]
the \emph{Teichm\"{u}ller codimension} of $\Gamma<\Delta$. If both groups have
genus zero then $d(\Gamma,\Delta)=\left\vert \mathcal{G}_{1}\right\vert
-\left\vert \mathcal{G}_{2}\right\vert $.
\end{itemize}

\section{The sequence $\Gamma_{C}\hookrightarrow\Gamma_{N}\twoheadrightarrow
K$ \label{sec-top}}

First we consider any exact sequence $\Gamma_{C}\hookrightarrow\Gamma
_{N}\twoheadrightarrow K$ where we only assume $\Gamma_{C}\trianglelefteq
\Gamma_{N}$ is a pair of genus zero, finite area Fuchsian groups. We are not
assuming any map $\Gamma_{C}\twoheadrightarrow C$. The induced map
$\chi:\Gamma_{N}\rightarrow K$ is called a $K\ $map.

\begin{definition}
Given an exact sequence $\Gamma_{C}\hookrightarrow\Gamma_{N}\twoheadrightarrow
K$ arising from a pair $\Gamma_{C}\trianglelefteq\Gamma_{N}$ of genus zero,
finite area Fuchsian groups, we say that a canonical generator $\theta
\in\Gamma_{N}$ is a $K$-generator if it has non-trivial image under the map
$\chi\colon\Gamma_{N}\rightarrow\Gamma_{N}/\Gamma=K$.
\end{definition}

\begin{proposition}
\label{pr-images1} Let $\Gamma_{C}\hookrightarrow\Gamma_{N}\twoheadrightarrow
K$ be any exact sequence defined by a pair $\Gamma_{C}\trianglelefteq
\Gamma_{N}$ of genus zero, finite area Fuchsian groups. Then, $K$ is a group
acting on the sphere with signature given in Theorem \ref{th-sphere}. The
images of the canonical generators of $\Gamma_{N}$ under the map $\chi
\colon\Gamma_{N}\rightarrow K$ satisfy the relations of Theorem \ref{th-sig}
for the signature of $K$. In particular, if $K$ is not trivial there are
exactly 2 $(K=C_{k})$ or 3 $(K\neq C_{k})$ canonical generators for
$\Gamma_{N}$ with non-trivial image under $\chi$.
\end{proposition}

By Remark \ref{rk-braid} we may permute the periods of $\Gamma_{N}$ so that
the $K$-generators occur first and the signature has the format\ $(m_{1}%
,\dots,m_{s})=(a_{1}b_{1},$ $a_{2}b_{2},$ $a_{3}b_{3},$ $m_{4},\dots,m_{s})$
if $K$ has signature $(a_{1},a_{2},a_{3})$ and $(m_{1},\dots,m_{s}) $ $=$
$(a_{1}b_{1},a_{2}b_{2},$ $m_{3},\dots,m_{s})$ $=$ $(kb_{1},kb_{2},$
$m_{3},\dots,m_{s})$ if $K$ has signature $(a_{1},a_{2})=(k,k).$ If
$\Gamma_{N}$ has a signature of either form, after permutation, we say that
$\Gamma_{N}$ has a $K$-compatible signature. We have the following converse to
Proposition \ref{pr-images1}, which follows directly from Lemma $5.8$ of
\cite{Brou-Woo1.5}.

\begin{proposition}
\label{pr-images2} Let $\left\{  \theta_{1},\dots,\theta_{s}\right\}  $ be a
set of canonical generators corresponding to the $K$-compatible signature
$(a_{1}b_{1},a_{2}b_{2},a_{3}b_{3},$ $m_{4},\dots,m_{s}).$ Then, there is an
essentially unique epimorphism $\chi:\Gamma_{N}\rightarrow K$ such that
$(x_{1},x_{2},x_{3})=$ $(\chi(\theta_{1}),\chi(\theta_{2}),\chi(\theta_{3}))$
is a generating $(a_{1},a_{2},a_{3})$-vector of $K$. I.e., given two
epimorphisms $\chi_{1}\colon\Gamma_{N}\rightarrow K,\chi_{2}\colon\Gamma
_{N}\rightarrow K$ such that $(\chi_{s}(\theta_{1}),\chi_{s}(\theta_{2}%
),\chi_{s}(\theta_{3}))$ are $(a_{1},a_{2},a_{3})$-vectors for $s=1,2$ then
$\chi_{2}=\omega\circ\chi_{1}$ for some $\omega\in\mathrm{Aut}(K).$ A similar
statement holds for the cyclic case.
\end{proposition}

\begin{remark}
Suppose we are given a generating $(a_{1},\ldots,a_{e})$-vector $(x_{1}%
,\ldots,x_{e})$ of $K.$ Then a $K$ map $\chi:\Gamma_{N}\rightarrow K $ may be
defined by%
\begin{equation}
\chi(\theta_{i})=x_{i},\text{ }1\leq i\leq e,\text{ }\chi(\theta_{i})=1,\text{
}e+1\leq i\leq s.\label{eq-Kmap}%
\end{equation}
Once the factorization $(a_{1}b_{1},\ldots,$ $a_{e}b_{e},m_{e+1},\dots,m_{s})$
is fixed then the kernel $\Gamma_{C}$ of the associated sequence $\Gamma
_{C}\hookrightarrow\Gamma_{N}\twoheadrightarrow K$ is unique.
\end{remark}

\subsection{\label{subsec-findseq}Finding $\Pi\hookrightarrow\Gamma
_{C}\twoheadrightarrow C$}

We assume that our $K$ map $\chi:\Gamma_{N}\rightarrow K$ is given as in
equation \ref{eq-Kmap}. We want to know when a $K$ map arises from the
normalizer of a cyclic $n$-gonal action. To this end let us denote by
$\Gamma_{C}$ the kernel of $\chi$ so that we have an exact sequence of the
form.%
\[%
\begin{array}
[c]{ccccc}%
\Gamma_{C} & \hookrightarrow & \Gamma_{N} & \overset{\mathcal{\chi}%
}{\twoheadrightarrow} & K.
\end{array}
\]
Let $\left\{  \xi_{1},\ldots,\xi_{r}\right\}  $ be an ordered set of canonical
generators for $\Gamma_{C}$. The canonical generators of $\Gamma_{C}$ are in
1-1 correspondence to the branch points $\mathbb{H\rightarrow H}/\Gamma_{C}$
and $K$ permutes these branch points. The canonical generators of $\Gamma_{N}$
give rise to $K$-orbits of $C$ branch points as follows. The $K$-generators
correspond to orbits of size less than $\left\vert K\right\vert $ and the
other orbits are regular $K$-orbits. It follows that the branch points of
$\Gamma_{C}$ are: $\left\vert K\right\vert /a_{i}$ branch points of period
$b_{i}$ (unless $b_{i}=1)$ for each $i,$ $1\leq i\leq e,$ (singular
$K$-orbits) and $\left\vert K\right\vert $ branch points of period $m_{j}$ for
each $j,$ $e+1\leq j\leq s$ (regular $K$-orbits). Next we need a map
$\phi:\Gamma_{C}\rightarrow C$ where $C$ is a cyclic group such that $\Pi=$
$\ker\phi$ is torsion free. Define $z_{i}\in C$ by
\begin{equation}
z_{i}=\phi(\xi_{i})\label{eq-zi}%
\end{equation}
so that $(z_{1},\dots,z_{r})$ is a generating $\mathcal{S}\left(  \Gamma
_{C}\right)  $-vector for the $C$-action. According to \cite{Har}, in order
that the vector exist and $\Pi$ be torsion free, we must have:

\begin{itemize}
\item $%
{\displaystyle\prod\limits_{i=1}^{r}}
z_{i}=1,$

\item $o(\xi_{i})=o(z_{i}),$

\item $\left\vert C\right\vert =\operatorname{lcm}(o(\xi_{1}),\ldots,o(\xi
_{r}))=\operatorname{lcm}(b_{1},\ldots,b_{e},m_{e+1},\ldots,m_{s}),$

\item some additional constraints on the periods $o(\xi_{1}),\ldots,o(\xi
_{r})$ given in Harvey's work \cite{Har}.
\end{itemize}

\noindent We now fix $C$ to have order $\operatorname{lcm}(b_{1},\ldots
,b_{e},m_{e+1},\ldots,m_{s}),$ and assume the constraints in the fourth bullet
above. Then, the set
\begin{equation}
X=\left\{  (z_{1},\dots,z_{r}):o(z_{i})=o(\xi_{i}),%
{\displaystyle\prod\limits_{i=1}^{r}}
z_{i}=1\right\} \label{eq-defofX}%
\end{equation}
of generating $\mathcal{S}\left(  \Gamma_{C}\right)  $-vectors is non-empty.
The set $X$ allows us to enumerate the epimorphisms $\phi:\Gamma
_{C}\rightarrow C$ since $\phi\rightarrow(\phi(\xi_{1}),\ldots,\phi(\xi_{r}))$
is a 1-1 correspondence. The group $\mathrm{Aut}(C)$ acts without fixed points
on the epimorphisms by $(\omega,\phi)\rightarrow\omega\circ\phi,$ this action
is transferred to $X$ by $(\omega,(z_{1},\dots,z_{r}))\rightarrow(\omega
(z_{1}),\dots,\omega(z_{r})).$ The possible kernels $\Pi\ $are in 1-1
correspondence with the $\mathrm{Aut}(C)$ orbits on $X,$ a finite computable set.

Next we need to determine when the homomorphism $\phi$ extends to a
homomorphism $\psi:\Gamma_{N}\rightarrow\widetilde{K}$ such that

\begin{itemize}
\item $\widetilde{K}$ is an overgroup of $C$ such that $C\vartriangleleft
\widetilde{K}$ and $\widetilde{K}/C\backsimeq K$

\item $\psi$ restricted to $\Gamma_{C}$ is $\phi:\Gamma_{C}\rightarrow C$
\end{itemize}

\noindent The group $\widetilde{K}$ will equal $N$ when identified with a
subgroup of $A$. To show that the two bullets hold, it suffices to show that
$\Pi$ is normal in $\Gamma_{N},$ for then we may take, abstractly,
$\widetilde{K}=\Gamma_{N}/\Pi.$ It does not give us $\widetilde{K}$ concretely
but suffices to show the extendability $\psi:\Gamma_{N}\rightarrow
\widetilde{K}.$

To find restrictions on the $z_{i}$ that will guarantee that $\Pi$ is normal
in $\Gamma_{N},$ we shall employ the methods in \cite{Brou-Woo1}. For any
$x\in\Gamma_{N}$ define $\phi_{x}:\Gamma_{C}\rightarrow C$ by $\phi_{x}%
(\gamma)=\phi(x\gamma x^{-1}).$ The kernel of $\phi_{x}$ is $x^{-1}\Pi x$ and
hence $\Pi$ is normal in $\Gamma_{N}$ if and only $\mathrm{\ker}\left(
\phi_{x}\right)  =\Pi$ for all $x\in\Gamma_{N}.$ But $\phi_{x}$ and $\phi$
have the same kernel if and only if there is an $\omega_{x}\in\mathrm{Aut}(C)$
such that $\phi_{x}=\omega_{x}\circ\phi$ or
\begin{equation}
\phi(x\gamma x^{-1})=\phi_{x}(\gamma)=\omega_{x}\left(  \phi(\gamma)\right)
,\gamma\in\Gamma_{C}.\label{eq-omega}%
\end{equation}
We then have for $x,y\in\Gamma_{N}$ and $\gamma\in\Gamma_{C}$
\begin{align*}
\omega_{xy}\left(  \phi(\gamma)\right)   &  =\phi_{xy}(\gamma)\\
&  =\phi(xy\gamma y^{-1}x^{-1})\\
&  =\omega_{x}(\phi(y\gamma y^{-1}))\\
&  =\omega_{x}(\omega_{y}(\phi(\gamma))
\end{align*}
and so $\omega_{xy}=\omega_{x}\circ\omega_{y},$ and thus $x\rightarrow
\omega_{x}$ is a homomorphism $\Gamma_{N}\rightarrow\mathrm{Aut}(C).$ Since
$C$ is abelian then $\omega_{x}=id$ for $x\in\Gamma_{C}$ and $x\rightarrow
\omega_{x}$ factors through $K,$ $g\rightarrow\omega_{g},$ $g\in K$.

\begin{remark}
\label{rk-ab}Observe that the homomorphisms $K\rightarrow\mathrm{Aut}(C)$ are
quite limited, since $\mathrm{Aut}(C)$ is abelian. Thus $\omega:K\rightarrow
\mathrm{Aut}(C)$ factors through the abelianization $\omega:K_{ab}%
\rightarrow\mathrm{Aut}(C).$ The abelianizations are given in Table 2.%
\[%
\begin{tabular}
[c]{c}%
$\text{\textbf{Table }}$\textbf{2}\\
$%
\begin{tabular}
[c]{||l|l|l||}\hline\hline
Group & Signature & Abelianization\\\hline
$C_{k}$ & $(k,k),k\geq2$ & $\mathbb{Z}_{k}$\\\hline
$D_{k}$ & $(2,2,k),k\geq2,$ $k$\textrm{\ even} & $\mathbb{Z}_{2}%
\times\mathbb{Z}_{2}$\\\hline
$D_{k}$ & $(2,2,k),k\geq3,$ $k$ \textrm{odd} & $\mathbb{Z}_{2}$\\\hline
$A_{4}$ & $(2,3,3)$ & $\mathbb{Z}_{2}\times\mathbb{Z}_{2}$\\\hline
$S_{4}$ & $(2,3,4)$ & $\mathbb{Z}_{2}$\\\hline
$A_{5}$ & $(2,3,5)$ & $\left\langle id\right\rangle $\\\hline\hline
\end{tabular}
$%
\end{tabular}
\]

\end{remark}

\begin{remark}
\label{rk-autCord2}From Table 2 we see that for the non-cyclic case we only
need to consider automorphisms of order 2. Let us write these down. For a
cyclic group $\mathbb{Z}_{n}$ the automorphism group is the group of units
$\mathbb{Z}_{n}^{\ast}$ which in turn is given by $\mathbb{Z}_{n}^{\ast}=%
{\displaystyle\prod\limits_{j}}
\mathbb{Z}_{p_{j}^{e_{j}}}^{\ast}$ where $n=%
{\displaystyle\prod\limits_{j}}
p_{j}^{e_{j}},$ since the Sylow subgroups are cyclic and invariant. The
automorphisms of order dividing 2, and their fixed point subgroups are
important to our analysis in Section \ref{sec-samp}. These automorphisms are
given by $x\rightarrow ax$ where $a^{2}=1\operatorname{mod}n.$ According to
the above decompositions we just need to determine the automorphisms for
$n=p^{e}$ a prime power. The automorphisms and their fixed points for the
various prime power cases are given in the Table 3.
\[%
\begin{tabular}
[c]{c}%
$\text{\textbf{Table }}$\textbf{3}\\
$%
\begin{tabular}
[c]{||l|l|l||}\hline\hline
$p^{e}$ & $a$ & fixed point subgroup\\\hline
$\mathrm{odd}$ $p$ & $1$ & $\mathbb{Z}_{p^{e}}$\\\hline
$\mathrm{odd}$ $p$ & $-1$ & $\left\langle 0\right\rangle $\\\hline
$2^{e},e\geq1$ & $1$ & $\mathbb{Z}_{2^{e}}$\\\hline
$2^{e},e\geq2$ & $-1$ & $2^{e-1}\mathbb{Z}_{2^{e}}$\\\hline
$2^{e},e\geq3$ & $2^{e-1}+1$ & $2\mathbb{Z}_{2^{e}}$\\\hline
$2^{e},e\geq3$ & $2^{e-1}-1$ & $2^{e-1}\mathbb{Z}_{2^{e}}$\\\hline\hline
\end{tabular}
$%
\end{tabular}
\]
The results in the table are derived by considering $p^{e}|(a-1)(a+1).$
\end{remark}

Now let us compute the corresponding action of $K$ on $X.$ For any canonical
generator $\xi_{i},$ $x\xi_{i}x^{-1}$ is an elliptic element of $\Gamma_{C}$
and hence belongs to $y\left\langle \xi_{j}\right\rangle y^{-1}$ for some
canonical generator $\xi_{j}$ and $y\in\Gamma_{C},$ by 3 of Theorem
\ref{th-sig}. Since both $x\xi_{i}x^{-1}$ and $y\xi_{j}y^{-1}$ generate the
stabilizer of the same point then $x\xi_{i}x^{-1}=y\xi_{j}^{a}y^{-1}$ where
$a$ is relatively prime to the order of $\xi_{j}.$ By using covering space
methods to construct the generating set $\left\{  \xi_{1},\ldots,\xi
_{r}\right\}  ,$ it can be shown that we may in fact take $a=1$ and that the
permutation representation $q:i\rightarrow j$ is defined by the action of $K$
on the branch points of $\mathbb{H}\rightarrow\mathbb{H}/\Gamma_{C}.$ We then
have
\begin{align*}
\omega_{x}(z_{i}) &  =\phi_{x}(\xi_{i})\\
&  =\phi(x\xi_{i}x^{-1})\\
&  =\phi(y\xi_{j}y^{-1})\\
&  =\phi(\xi_{j})\\
&  =z_{j}%
\end{align*}
as $y\in\Gamma_{C}.$ We piece together the data above to construct an action
of $K$ on $X$ by
\begin{equation}
g\cdot(z_{1},\dots,z_{r})=(\omega_{g^{-1}}(z_{q_{g}(1)}),\dots,\omega_{g^{-1}%
}(z_{q_{g}(r)}))\label{eq-Kaction}%
\end{equation}
The vector $(z_{1},\dots,z_{r})$ is fixed by $g$ if and only if
\begin{equation}
\omega_{g}(z_{i})=z_{q(i)},\ 1\leq i\leq r.\label{eq-Kactioni}%
\end{equation}

The following theorem allows us to identify normalizers of cyclic $n$-gonal
actions, by finding the $K$-fixed points of the actions in equation
\ref{eq-Kaction} as we vary over all homomorphisms $K\rightarrow
\mathrm{Aut}(C)$. The proof of the theorem follows from the previous discussion.

\begin{theorem}
Let the sequence $\Gamma_{C}\hookrightarrow\Gamma_{N}\twoheadrightarrow K, $
the cyclic group $C,$ the set of generating vectors $X,$ and the permutation
representation $q:K\rightarrow\Sigma_{r}$ be as defined above. Then we have
the following.

\begin{itemize}
\item Let $(z_{1},\dots,z_{r})\in X$ be a generating $\mathcal{S}(\Gamma_{C}%
)$-vector of $C$ and $\Pi\rightarrow\Gamma_{C}\overset{\phi}{\rightarrow}C$
the epimorphism sequence defined by $\phi(\xi_{i})=z_{i}.$ Assume that $\Pi$
is normal in $\Gamma_{N},$ that $\omega:$ $K\rightarrow\mathrm{Aut}(C)$ is the
resulting homomorphism defined by equation \ref{eq-omega}, and let $K$ act on
$X$ by equation \ref{eq-Kaction}. Then $(z_{1},\dots,z_{r})$ is fixed by all
$g$ in $K.\ $

\item Let $\omega:K\rightarrow\mathrm{Aut}(C)$ be any homomorphism, and let
$K$ act on $X$ by equation \ref{eq-Kaction}$.$ Assume that $(z_{1},\dots
,z_{r})\in X$ is fixed by all $g$ in $K$ and let $\Pi\rightarrow\Gamma
_{C}\overset{\phi}{\rightarrow}C$ be the epimorphism sequence defined by
$\phi(\xi_{i})=z_{i}.$ Then $\Pi$ is normal in $\Gamma_{N}.$
\end{itemize}
\end{theorem}

\begin{example}
Let $\Gamma_{N}$ have signature $(4,4,9,11)$ written in factored form as
$(2\cdot2,2\cdot2,3\cdot3,11)$ where $K=D_{3}$ has signature $(2,2,3).$ Then
$\Gamma_{C}$ has signature $(2^{3},2^{3},3^{2},11^{6}).$ We determine all
possible sequences
\[%
\begin{array}
[c]{ccccc}%
\Gamma_{C} & \hookrightarrow & \Gamma_{N} & \overset{\mathcal{\chi}%
}{\twoheadrightarrow} & K
\end{array}
\]
with $C$ cyclic. As noted in Table 3 $K_{ab}=\mathbb{Z}_{2}$ and we really
only need to carefully consider the action of the reflections in $K.$

Let us first discuss the action of $K$ on the indices $\{1,\dots,14\}.$ This
action is derived from the $K$-action on the sphere, so we just need to
describe it one orbit at a time. The indices $\{1,2,3\}$ correspond to one of
the orbits of size three and the $K$-action is just the standard $D_{3}%
$-action. Likewise for the indices $\{4,5,6\}.$ The indices $\{7,8\}$
correspond to the orbit of size two and so the reflections in $K$ interchange
$7$ and $8.$ Finally $\{9,10,11,12,13,14\}$ constitutes a regular orbit and so
we may arrange the indices so that the reflections in $K $ interchange
$\{9,10,11\}$ and $\{12,13,14\}$ as sets.

Now let $C=C_{66}=C_{2}\times C_{3}\times C_{11}$ and from Remark
\ref{rk-autCord2} $\mathrm{Aut}(C)\backsimeq\mathbb{Z}_{2}\times
\mathbb{Z}_{10}.$ Define $g_{2},g_{3},g_{11},$ so that $C_{2}=\left\langle
g_{2}\right\rangle ,$ $C_{3}=\left\langle g_{3}\right\rangle ,$ $C_{11}%
=\left\langle g_{11}\right\rangle .$ Since the abelianization $K_{ab}%
\backsimeq\mathbb{Z}_{2},$ the image $\omega_{r}$ of the non-trivial element
of $K_{ab}$ must be in the subgroup of $\mathrm{Aut}(C)$ generated by
$\omega_{1}:(x,y,z)\rightarrow(x,y^{-1},z)$ and $\omega_{2}:(x,y,z)\rightarrow
(x,y,z^{-1}),$ for $(x,y,z)\in C_{2}\times C_{3}\times C_{11}$.

Now let us consider a specific map $\omega_{r}\in\mathrm{Aut}(C)$ and a
specific vector. Set $\omega_{r}=\omega_{1}\omega_{2}$ and consider the
following vector
\begin{equation}
(z_{1},\dots,z_{14})=(g_{2},g_{2},g_{2},g_{2},g_{2},g_{2},g_{3},g_{3}%
^{-1},g_{11},g_{11},g_{11},g_{11}^{-1},g_{11}^{-1},g_{11}^{-1}%
).\label{eq-gv66}%
\end{equation}
By construction $(z_{1},\dots,z_{14})$ satisfies equations \ref{eq-gv1},
\ref{eq-gv2}, \ref{eq-gv3}, and it is also fixed by $K$ under the action given
by equation \ref{eq-Kaction} or equation \ref{eq-Kactioni}. Thus the action of
\ $C$ may be extended by $K.$ $\ $The given vector is essentially unique.
First we can only have $\omega_{r}=\omega_{1}\omega_{2}.$ If $\omega_{r}$ acts
trivially on $C_{3}\subset C$ then we must have $z_{8}=\omega_{r}(z_{7}%
)=z_{7}$ by equation \ref{eq-Kactioni}$\ $but then
\[
1=z_{1}\cdot\cdots\cdot z_{14}=z_{1}\cdot\cdots\cdot z_{6}z_{7}^{2}z_{9}%
\cdot\cdots\cdot z_{14}%
\]
Since $z_{7}^{2}$ has order 3, it is not possible for this product to be
trivial no matter what the values of other $z_{i}$ are. Likewise $\omega_{r} $
cannot act trivially on $C_{11}$ otherwise $z_{9}=\cdots=z_{14},$ by equation
\ref{eq-Kactioni}, yielding another contradiction. Now that $\omega_{r}$ is
determined it follows that we can only have a vector of the form in equation
\ref{eq-gv66}, where $g_{2},g_{3},g_{11}$ are suitably chosen generators. Any
two such vectors are equivalent under $\mathrm{Aut}(C).$
\end{example}

\begin{remark}
\label{rk-Kfixed}The above example demonstrates the following easily proved
properties of $K$-fixed-vectors. Using the properties allows us to easily
construct and enumerate the $K$-fixed vectors.

\begin{enumerate}
\item The element $z_{i}$ must be invariant under $\left\{  \omega_{g}%
:q_{g}(i)=i\right\}  $.

\item The collection of $z_{i}$ corresponding to a $K$-orbit $\mathcal{O}%
\subseteq\{1,\ldots,r\}$ are determined by a single $z_{i_{0}}$ for any
$i_{0}\in\mathcal{O}.$ Just use the $K$-action.

\item If $\mathcal{O}_{1},\ldots,\mathcal{O}_{s}$ are the $K$-orbits
constituting $\{1,\ldots,r\}$ Then
\[%
{\displaystyle\prod\limits_{j=1}^{s}}
\left(
{\displaystyle\prod\limits_{i\in\mathcal{O}_{j}}}
z_{i}\right)  =1.
\]

\end{enumerate}

Now one simply finds an element in each orbit of the correct order satisfying
statement 1$.$ By statement 2 the orbit products $%
{\textstyle\prod\limits_{i\in\mathcal{O}_{j}}}
z_{i}$ are easily calculated and we just have to verify statement 3$.$
\end{remark}

\subsection{Finding epimorphisms $\Gamma_{N}\twoheadrightarrow K$}

By proposition \ref{pr-images2} there are epimorphisms $\Gamma_{N}%
\twoheadrightarrow K$ if and only if $\Gamma_{N}$ has a\ $K$-compatible
signature. Finding maps $\Gamma_{N}\rightarrow K$ is fairly simple when all
the periods are known constants. However, as we shall see in the next section,
we want to consider that case when the periods are parameters, such as in
Example \ref{ex-parametric}. There needs to be some care to get an efficient
enumeration of all the cases. We first consider an example.

\begin{example}
\label{ex-triangle}Suppose that $\Gamma_{N}=T(2,d,2d)$ with signature
$(2,d,2d),d\geq4$ We want to permute and factor the signature $(2,d,2d)$ so
that it is $K$-compatible, i.e., in the form $(a_{1}b_{1},a_{2}b_{2}%
,a_{3}b_{3})$ or $(kb_{1},kb_{2},m_{1})$ We put the results in Table 4 below.
In the factorizations the variable $e$ may be any integer such that the
signature $(2,d,2d)$ is hyperbolic. Some factorizations are equivalent by
permutations that leave the signature of $K$ fixed, they are listed
contiguously. The signature of $C$ and the $n=\left\vert C\right\vert $ can be
computed as at the beginning of subsection \ref{subsec-findseq}.
\[%
\begin{tabular}
[c]{c}%
\textbf{Table 4 - part 1}\\
$%
\begin{tabular}
[c]{||l|l|l|l||}\hline\hline
$K$ & $\mathcal{S}(K)$ & $\mathcal{S}(\Gamma_{N})$ factored & conditions on
$d,e,k,n$\\\hline
$C_{2}$ & $(2,2)$ & $(2,d,2d)=(2\cdot1,2\cdot e,4e)$ & $d=2e,n=4e$\\\hline
$\ $ & $(2,2)$ & $(d,2,2d)=(2\cdot e,2\cdot1,4e)$ & $d=2e,n=4e$\\\hline
& $(2,2)$ & $(2,2d,d)=(2\cdot1,2\cdot e,e)$ & $n=e=d,$ $e$ \textrm{odd,
}$e\geq5$\\\hline
& $(2,2)$ & $(2d,2,d)=(2\cdot e,2\cdot1,e)$ & $n=e=d,$ $e$ \textrm{odd,
}$e\geq5$\\\hline
$C_{k}$ & $(k,k)$ & $(d,2d,2)=(k\cdot e,k\cdot2e,2)$ & $d=ke,n=2e,$ $d$
\textrm{odd}\\\hline
& $(k,k)$ & $(2d,d,2)=(k\cdot2e,k\cdot e,2)$ & $d=ke,n=2e,$ $d$ \textrm{odd}%
\\\hline
$D_{2}$ & $(2,2,2)$ & $(d,2d,2)=(2\cdot e,2\cdot2e,2\cdot1)$ & $d=2e,n=2e$%
\\\hline
& $(2,2,2)$ & $(2d,d,2)=(2\cdot2e,2\cdot e,2\cdot1)$ & $d=2e,n=2e$\\\hline
$D_{k}$ & $(2,2,k)$ & $(2,2d,d)=(2\cdot1,2\cdot ke,k\cdot e)$ & $d=n=ke$%
\\\hline
&  &  & $d$ \textrm{odd or }$k$\textrm{\ even}\\\hline
& $(2,2,k)$ & $(2d,2,d)=(2\cdot ke,2\cdot1,k\cdot e)$ & $d=n=ke$\\\hline
&  &  & $d$ \textrm{odd or }$k$\textrm{\ even}\\\hline
& $(2,2,k)$ & $(2,d,2d)=(2\cdot1,2\cdot\frac{d}{2},k\cdot e)$ & $2d=ke,$ $d $
\textrm{even,}\\\hline
&  &  & $n=\operatorname{lcm}\left(  e,\frac{ek}{4}\right)  $\\\hline
& $(2,2,k)$ & $(d,2,2d)=(2\cdot\frac{d}{2},2\cdot1,k\cdot e)$ & $2d=ke,$ $d $
\textrm{even,}\\\hline
&  &  & $n=\operatorname{lcm}\left(  e,\frac{ek}{4}\right)  $\\\hline\hline
\end{tabular}
$%
\end{tabular}
\]%
\[%
\begin{tabular}
[c]{c}%
\textbf{Table 4 - part 2}\\
$%
\begin{tabular}
[c]{||l|l|l||l||}\hline\hline
$K$ & $\mathcal{S}(K)$ & $\mathcal{S}(\Gamma_{N})$ factored & conditions on
$d,e,n$\\\hline
$A_{4}$ & $(2,3,3)$ & $(2,d,2d)=(2\cdot1,3\cdot e,3\cdot2e)$ & $d=3e,n=2e$%
\\\hline
&  & $(2,2d,d)=(2\cdot1,3\cdot2e,3\cdot e)$ & $d=3e,n=2e$\\\hline
$S_{4}$ & $(2,3,4)$ & $(2,d,2d)=(2\cdot1,3\cdot2e,4\cdot3e)$ & $d=n=6e$%
\\\hline
&  & $(2,2d,d)=(2\cdot1,3\cdot8e,4\cdot3e)$ & $d=12e,n=24e$\\\hline
$A_{5}$ & $(2,3,5)$ & $(2,d,2d)=(2\cdot1,3\cdot5e,5\cdot6e)$ & $d=15e,n=30e$%
\\\hline
&  & $(2,2d,d)=(2\cdot1,3\cdot10e,5\cdot3e)$ & $d=15e,n=30e$\\\hline\hline
\end{tabular}
$%
\end{tabular}
\]
\newline\newline We prove a few of the lines. \newline\newline Line 1: Since
$S(\Gamma_{N})=$ $(2,d,2d),$ $K=C_{2},$ the signature of $C$ is $(1,\frac
{d}{2},\left(  2d\right)  ^{k}),$ setting $d=2e$, $k=2$ we get $\mathcal{S}%
(C)=(e,4e,4e).$ According to \cite{Har}, $n=4e$ and a \ $C$-action exists on a
surface of genus $2e-2$ provided $e\geq2$.\newline\newline Line 3: The
signature of $C$ is $(1,e,e^{k})$ or $(e,e,e).$ A cyclic action with $n=e$ on
a surface of genus $\sigma=(e-1)/2$ exists if $e\geq3$ and $e $ is
odd.\newline\newline Line 5: The signature of $C$ is $(e,2e,2^{k}).$ If either
$e$ or $k$ is even, then the number of periods divisible by the highest power
of 2 is odd, violating one the conditions in \cite{Har}. Thus $d$ is odd and
$e\geq3.$ Harvey's conditions now hold and $n=2e$, $\sigma=\frac{d-1}{2}%
.$\newline\newline Line 7: The signature of $C$ is $(e^{2},d^{2},1^{2})$ or
$(e,e,2e,2e).$ A cyclic action with $n=2e$ exists on a surface of genus $2e-2$
if $e\geq2$.\newline\newline Line 9: The signature of $C$ is $(1^{k}%
,d^{k},\left(  \frac{d}{k}\right)  ^{2}) $ or $(e,e,\left(  ek\right)  ^{k}).$
We must have $n=d$ and $\sigma=\frac{k(d-3)+2}{2}.$ If $d$ is odd or if $k$ is
even then the signature meets the parity conditions in \cite{Har}, and an
action exists.\newline\newline Lines 13: The signature of $C$ is
$(1^{k},\left(  \frac{d}{2}\right)  ^{k},\left(  \frac{2d}{k}\right)  ^{2})$
or $(e,e,\left(  \frac{ek}{4}\right)  ^{k}),$ upon setting $2d=ek.$ We must
have $n=\operatorname{lcm}\left(  e,\frac{ek}{4}\right)  \ $ and
$n=ek,\frac{ek}{2},$ or $\frac{ek}{4}$ are all possible. The genus is and
$\sigma=1+\frac{nk}{2}-\frac{4n}{e}.$
\end{example}

Now we describe an algorithm for generating all possible maps $\Gamma
_{N}\rightarrow K,$ or equivalently the compatible, permuted signatures. See
Example \ref{ex-GammaNtoK-algorithm} for various steps of the process. First
we build a list of all possibilities and then prune the list to remove redundancies.

\begin{enumerate}
\item Enumerate all distinct ordered pairs $(l_{1},l_{2})$ or triples of
periods $(l_{1},l_{2},l_{3})$ from the periods of $\Gamma_{N},$ depending on
whether $K$ has 2 or 3 canonical generators.

\item Rewrite the periods in the form $(l_{1},l_{2},m_{1},\dots,m_{u})$ or
$(l_{1},l_{2},l_{3},m_{1},\dots,m_{u})$ so that the ordered pair or triple
occurs first and the remaining periods are ordered lexicographically with
respect to parameter variables, using increasing order on the coefficients.

\item For permuted $\mathcal{S}(\Gamma_{N})$ found in step 2 we solve
$(l_{1},l_{2})=(kb_{1},kb_{2}),\ $ $(l_{1},l_{2},l_{3})$ $=$ $(a_{1}%
b_{1},a_{2}b_{2},a_{3}b_{3}).$ We split this into two cases depending on
whether the signature of $K$ has parameters or not. Initially the parametric
signatures are $(k,k)$ or $(2,2,k)$ but these may be changed later on.

\item If the signature of $K$ consist only of constants we proceed as follows.

\begin{itemize}
\item We examine each $a_{i}$ in order, modifying $\mathcal{S}(\Gamma_{N})$ as needed.

\item If $l_{i}$ is a constant not divisible by $a_{i}$ we reject the permuted
$\mathcal{S}(\Gamma_{N})$.

\item Otherwise write $l_{i}=c_{i}w_{i}$ where $c_{i}$ is a constant and
$w_{i}$ is a parameter. Set $e_{i}=a_{i}/\gcd(a_{i},c_{i})$ and make the
substitution $w_{i}\rightarrow e_{i}w_{i}$ throughout the signature. See
Example \ref{ex-GammaNtoK-algorithm}, item 1.
\end{itemize}

\item If the signature of $K$ has a parameter $(k,k)$ or $(2,2,k)$ we proceed
as follows.

\begin{itemize}
\item We examine each $a_{i}$ in order, modifying $\mathcal{S}(\Gamma_{N})$ as needed.

\item If $a_{i}$ is a constant then we proceed as in step 4.

\item If $a_{i}$ is a parameter and $l_{i}$ is a constant then for each
divisor $d$ of $l_{i},d>1,$ solve the problem with $\mathcal{S}(K)=(d,d)$ or
$(2,2,d)$ and $\mathcal{S}(\Gamma_{N}).$ See Example
\ref{ex-GammaNtoK-algorithm}, item 2.

\item If $a_{i}$ has a parameter and $l_{i}$ has a parameter then we modify
with a separate case for the dihedral and cyclic cases.

\item \textit{Cyclic case} $l_{1}=c_{1}w_{1},$ $l_{2}=c_{2}w_{2}$: Let $d$ be
any divisor of $\gcd(c_{1},c_{2})$ then set $\mathcal{S}(K)=(dk,dk)$ and make
the substitution $w_{i}\rightarrow kw_{i}$ for each distinct $w_{i}.$

\item \textit{Dihedral Case} $l_{3}=c_{3}w_{3}$: First modify $\mathcal{S}%
(\Gamma_{N})$ as in the first bullet, possibly getting a new equation
$l_{3}=c_{3}w_{3}$. Let $d$ be any divisor of $c_{3}$ then set $\mathcal{S}%
(K)=(2,2,dk)$ and make the substitution $w_{3}\rightarrow kw_{3}.$
\end{itemize}
\end{enumerate}

\begin{example}
\label{ex-GammaNtoK-algorithm} Here are some examples of steps in the
algorithm above. We denote the desired map $\chi:\Gamma_{N}\rightarrow K$ by
$\mathcal{S}(\Gamma_{N})/\mathcal{S}(K).$ Steps in the process corresponding
to period $a_{i}$ of $K$ are denoted by the numbered arrow $\overset
{i}{\longrightarrow}.$

\begin{enumerate}
\item First let $\mathcal{S}(\Gamma_{N})=(2,2,x_{1},5x_{1}),$ $\mathcal{S}%
(K)=(2,3,5).$ The 12 permutations of $\mathcal{S}(\Gamma_{N})$ to be
considered are%
\begin{align*}
&  (2,2,x_{1},5x_{1}),(2,x_{1},2,5x_{1}),(x_{1},2,2,5x_{1}),(2,2,5x_{1}%
,x_{1}),\\
&  (2,5x_{1},2,x_{1}),(5x_{1},2,2,x_{1}),(2,x_{1},5x_{1},2),(2,5x_{1}%
,x_{1},2),\\
&  (x_{1},2,5x_{1},2),(5x_{1},2,x_{1},2),(x_{1},5x_{1},2,2),(5x_{1}%
,x_{1},2,2).
\end{align*}
If we consider the case $(2,x_{1},5x_{1},2),$ then the sequence of
substitutions required is:%
\begin{align*}
&  (2,x_{1},5x_{1},2)/(2,3,5)\overset{1}{\longrightarrow}(2,x_{1}%
,5x_{1},2)/(2,3,5)\overset{2}{\longrightarrow}\\
&  (2,3x_{1},15x_{1},2)/(2,3,5)\overset{3}{\longrightarrow}(2,3x_{1}%
,15x_{1},2)/(2,3,5)
\end{align*}

\item Let $\mathcal{S}(\Gamma_{N})=(6,x_{1},5x_{1},6),$ $\mathcal{S}(K)=(k,k)
$ From $(k,k)=(2,2)$ we get the sequence of substitutions required
\[
(6,x_{1},5x_{1},6)/(2,2)\overset{1}{\longrightarrow}(6,x_{1},5x_{1}%
,6)/(2,2)\overset{2}{\longrightarrow}(6,2x_{1},10x_{1},6)/(2,2)
\]
and from $(k,k)=(3,3)$ we get
\[
(6,x_{1},5x_{1},6)/(3,3)\overset{1}{\longrightarrow}(6,x_{1},5x_{1}%
,6)/(3,3)\overset{2}{\longrightarrow}(6,3x_{1},15x_{1},6)/(3,3)
\]

\item Let $\mathcal{S}(\Gamma_{N})=(6x_{1},10x_{1},2,2),$ $\mathcal{S}%
(K)=(k,k).$ Then we get
\[
(6x_{1},10x_{1},2,2)/(k,k)\rightarrow(6kx_{1},10kx_{1},2,2)/(k,k))
\]
and
\[
(6x_{1},10x_{1},2,2)/(k,k)\rightarrow(6kx_{1},10kx_{1},2,2)/(2k,2k))
\]

\item Let $\mathcal{S}(\Gamma_{N})=(2,x_{1},5x_{1},2),$ $\mathcal{S}%
(K)=(2,2,k).$ Then we get
\begin{align*}
&  (2,x_{1},5x_{1},2)/(2,2,k)\overset{1}{\longrightarrow}(2,x_{1}%
,5x_{1},2)/(2,2,k)\overset{2}{\longrightarrow}\\
&  (2,2x_{1},10x_{1},2)/(2,2,k)\overset{3}{\longrightarrow}(2,2kx_{1}%
,10kx_{1},2)/(2,2,k)
\end{align*}
or
\begin{align*}
&  (2,x_{1},5x_{1},2)/(2,2,k)\overset{1}{\longrightarrow}(2,x_{1}%
,5x_{1},2)/(2,2,k)\overset{2}{\longrightarrow}\\
&  (2,2x_{1},10x_{1},2)/(2,2,k)\overset{3}{\longrightarrow}(2,2kx_{1}%
,10kx_{1},2)/(2,2,2k)
\end{align*}
or%
\begin{align*}
&  (2,x_{1},5x_{1},2)/(2,2,k)\overset{1}{\longrightarrow}(2,x_{1}%
,5x_{1},2)/(2,2,k)\overset{2}{\longrightarrow}\\
&  (2,2x_{1},10x_{1},2)/(2,2,k)\overset{3}{\longrightarrow}(2,2kx_{1}%
,10kx_{1},2)/(2,2,5k)
\end{align*}
or%
\begin{align*}
&  (2,x_{1},5x_{1},2)/(2,2,k)\overset{1}{\longrightarrow}(2,x_{1}%
,5x_{1},2)/(2,2,k)\overset{2}{\longrightarrow}\\
&  (2,2x_{1},10x_{1},2)/(2,2,k)\overset{3}{\longrightarrow}(2,2kx_{1}%
,10kx_{1},2)/(2,2,10k)
\end{align*}

\end{enumerate}
\end{example}

\section{The Fuchsian group pair $\Gamma_{N}<\Gamma_{A}$ \label{sec-bot}}

Since it is unlikely that $\Gamma_{N}$ is normal in $\Gamma_{A}$ we need to
find ways to work with the structure of the inclusion of the pair $\Gamma
_{N}<\Gamma_{A}.$ We shall describe two different approaches: monodromy and
word maps. Since these concepts require significant computational power to
fully implement we only discuss them very briefly and refer the reader to
\cite{Brou3} for full details. In our examples in Section \ref{sec-samp} we
shall use ad hoc methods to directly construct a candidate for the full
automorphism group. Then we will use ad hoc applications of the monodromy
group and a maximality result, discussed at the end of this section, to
demonstrate that the candidate is the full automorphism group. On the other
hand, the machinery of monodromy groups and word maps is necessary for full
classification and computing the harder examples. Thus, we include an overview
of those ideas to give a complete overview of the classification process.

In Singerman's paper \cite{Sing2} on finite maximality, the inclusions
$\Gamma_{N}<\Gamma_{A}$ where both $\Gamma_{N}$ and $\Gamma_{A}$ are triangle
groups were determined. These pairs constitute the main part of what is known
as \textquotedblleft Singerman's list". Later, the authors of \cite{Bujetal}
presented methods useful in finding $A,$ if it exists, given $N$ and
$\Gamma_{N}<\Gamma_{A}.$ However, their methods were restricted to pairs on
Singerman's list. As described in the signature theorem, Theorem
\ref{thm-sigsweak}, there may be pairs $\Gamma_{N}<\Gamma_{A}$ which do not
appear in Singerman's list. Hence, we need the more general discussion of
pairs $\Gamma_{N}<\Gamma_{A}$ given in this section.

\subsection{Monodromy and word maps\label{subsec-monowordmap}}

Let $\Gamma<\Delta$ be a finite index pair of genus zero Fuchsian groups and
let $m=[\Gamma:\Delta].$ Any labeling of the cosets of $\Gamma$ in $\Delta$
gives rise to a permutation representation $\rho:\Delta\rightarrow\Sigma_{m}.$
If another labeling is chosen then the two representations are related by
$\rho_{2}=\pi\rho_{1}\pi^{-1}$ for some $\pi\in\Sigma_{m}.$ Thus all the
images $\rho(\Delta)$ are conjugate and are isomorphic to $\Delta
/$\textrm{Core}$_{\Delta}(\Gamma).$ We call any of the images or $\Delta
/$\textrm{Core}$_{\Delta}(\Gamma)$ itself the monodromy group $M(\Delta
,\Gamma).$ The monodromy group $M(\Delta,\Gamma) $ is isomorphic to the
monodromy of the branched cover $\mathbb{H}/\Delta\rightarrow\mathbb{H}%
/\Gamma$ away from the branch points. Since the groups are genus zero
$\mathbb{H}/\Delta\rightarrow\mathbb{H}/\Gamma$ is just a branched covering of
the sphere by itself.

If $\ \mathcal{G}_{2}=\left\{  \zeta_{1},\dots,\zeta_{t}\right\}  $ is the
chosen set of canonical generators of $\Delta,$ then the permutations%
\[
\pi_{j}=\rho\left(  \zeta_{j}\right)
\]
satisfy
\[
\prod\limits_{j=1}^{t}\pi_{j}=1
\]
because of equation \ref{eq-rels}. The monodromy group $M(\Delta
,\Gamma)=\left\langle \pi_{1},\pi_{2},\ldots,\pi_{t}\right\rangle $ is a
transitive subgroup of $\Sigma_{m}.$

\begin{remark}
If $\Gamma=\Gamma_{N}$ and $\Delta=\Gamma_{A}$ then $\Delta/$\textrm{Core}%
$_{\Delta}(\Gamma)\backsimeq A/$\textrm{Core}$_{A}(N)$ and $\mathbb{H}%
/\Delta\rightarrow\mathbb{H}/\Gamma$ is the projection $S/N\rightarrow S/A.$
\end{remark}

\begin{definition}
Let notation be as above and set $\mathcal{P=(}\pi_{1},\ldots,\pi_{t}).$ The
cycle type of $\pi_{j}$ determines a partition $p_{j}$ of $m,$ set
$P=(p_{1},\ldots,p_{t})$. The tuple of permutations $\mathcal{P}$ is called
the \emph{monodromy vector} of $\Gamma<\Delta$ or $\mathbb{H}/\Delta
\rightarrow\mathbb{H}/\Gamma.$ The tuple of partitions $P$ is called the
\emph{cycle vector} of $\Gamma<\Delta$ or $\mathbb{H}/\Delta\rightarrow
\mathbb{H}/\Gamma.$ More generally, let $P=(p_{1},\ldots,p_{t})$ be a
$t$-tuple of partitions and let $\mathcal{P=(}\pi_{1},\ldots,\pi_{r})$ be
$t$-tuple of permutations. Then $\mathcal{P}$ is called a transitive
$P$-monodromy vector if
\begin{equation}
\pi_{j}\text{ \textrm{has cycle type} }p_{j}\label{eq-monovec1}%
\end{equation}%
\begin{equation}
\prod\limits_{j=1}^{t}\pi_{j}=1\label{eq-monovec2}%
\end{equation}%
\begin{equation}
\left\langle \pi_{1},\pi_{2},\ldots,\pi_{t}\right\rangle \text{ \textrm{is a
transitive subgroup of} }\Sigma_{m}.\label{eq-monovec3}%
\end{equation}

\end{definition}

\begin{remark}
The signatures $\mathcal{S}(\Gamma)$, $\mathcal{S}(\Delta)$ determine the
cycle types occurring in the cycle vector for $\Gamma<\Delta.$ Indeed, let
$p_{j}=(p_{j,1},\ldots,p_{j,s_{j}})$ be the partition of $n$ determined by
$\pi_{j}.$ Then for each $p_{j,i}$ there is a distinct generator $\theta
_{j,i}$ of $\Gamma$ of order $m_{j,i}$ such that
\begin{equation}
o(\zeta_{j})=p_{j,i}o(\theta_{j,i})\label{eq-compatible1}%
\end{equation}
or%
\begin{equation}
n_{j}=p_{j,i}m_{j,i}\label{eq-compatible2}%
\end{equation}
where $\mathcal{S}(\Delta)=(n_{1},\ldots,n_{t}).$ We say that the pair of
signatures $\mathcal{S}(\Gamma)<\mathcal{S}(\Delta)$ of signatures are
$P$-compatible, and symbolize this by
\[
P:\mathcal{S}(\Gamma)\rightarrow\mathcal{S}(\Delta)
\]
We call the sequence a numerical projection even though there may not be a
projection of surfaces $\pi:\mathbb{H}/\Gamma\rightarrow\mathbb{H}/\Delta.$
\end{remark}

The following variant of the Riemann existence theorem is important for our work.

\begin{theorem}
Let $\Gamma,\Delta$ be a two Fuchsian groups, and $P$ a cycle vector, and
suppose that the signatures $\mathcal{S}(\Gamma)$ and $\mathcal{S}(\Delta)$
are $P$-compatible. Let $\mathcal{P}$ be a transitive $P$-monodromy vector.
Then there is a subgroup $\Gamma^{\prime}<$ $\Delta,$ with the same signature
as $\Gamma,$ such that $\mathcal{P}$ is the monodromy vector of the pair
$\Gamma^{\prime}<\Delta.$
\end{theorem}

Now we turn our attention to word maps.

\begin{definition}
Select canonical generating sets $\mathcal{G}_{1}=\{\theta_{1},\dots\theta
_{s}\}$ and $\mathcal{G}_{2}=\{\zeta_{1},\dots\zeta_{t}\}$ of $\Gamma$ and
$\Delta$ respectively. The word map of the pair $\Gamma\leq\Delta$ is a set of
words $\{w_{1},\ldots,w_{s}\}$ in the generators in $\mathcal{G}_{2}$ such
that
\[
\theta_{i}=w_{i}(\zeta_{1},\ldots,\zeta_{t}),i=1,\ldots,s.
\]

\end{definition}

\begin{remark}
\label{rk-GZwordmap}If both groups have genus zero there is an easily
implemented algorithm to calculate the word map, see \cite{Brou3}. The word
maps for the inclusions in Singerman's list have been calculated in
\cite{Bujetal}.
\end{remark}

\begin{example}
Suppose we have the signatures $\mathcal{S}_{1}=(2,2,2,5),\mathcal{S}%
_{2}=(2,4,5).$ We want to show there is a pair $\Gamma<\Delta$ with
$\mathcal{S}(\Gamma)=\mathcal{S}_{1},\mathcal{S}(\Delta)=\mathcal{S}_{2}.$
First find a compatible monodromy vector $\mathcal{P=(}\pi_{1},\pi_{2},\pi
_{3})$ in $\Sigma_{6}.$ We select
\[
\pi_{1}=(1,3)(4,6),\pi_{2}=(1,2)(3,5,4,6),\pi_{3}=(1,2,3,4,5)
\]
from which we get $M(\Delta,\Gamma)=A_{6}.$ Define as before $\rho\colon
\Delta\rightarrow\Sigma_{6}$ by $\rho\colon\zeta_{i}\rightarrow\pi
_{i},i=1\ldots3.$ Then $\Gamma$ may be taken as the stabilizer of a point for
the permutation action of $\Delta$ on $\{1,\ldots,6\}.$ From the algorithm, a
generating set for $\Gamma$ is
\begin{align*}
\theta_{1} &  =(\zeta_{1}\zeta_{2})\zeta_{1}(\zeta_{1}\zeta_{2})^{-1}\\
\theta_{2} &  =\zeta_{2}\zeta_{1}\zeta_{2}^{-1}\\
\theta_{3} &  =\zeta_{2}^{2}\\
\theta_{4} &  =(\zeta_{2}^{-1}\zeta_{1}^{-1}\zeta_{2}^{-1}\zeta_{1}\zeta
_{3}\zeta_{1})\zeta_{3}(\zeta_{2}^{-1}\zeta_{1}^{-1}\zeta_{2}^{-1}\zeta
_{1}\zeta_{3}\zeta_{1})^{-1}%
\end{align*}

\end{example}

Here is how the word maps may be used in conjunction with the monodromy
vectors to expand an extension $C\vartriangleleft N$ to $C\vartriangleleft
N<A.$

\begin{itemize}
\item Assume that we have pairs $\Gamma_{N}<\Gamma_{A}$ and $\Gamma
_{C}\vartriangleleft\Gamma_{N}$ determined by monodromy groups $M(\Gamma
_{A},\Gamma_{N})$ and $M(\Gamma_{N},\Gamma_{C})\backsimeq K.$

\item According to Remark \ref{rk-GZwordmap} there are word maps for the
inclusions $\Gamma_{C}\vartriangleleft\Gamma_{N}$ and $\Gamma_{N}<\Gamma_{A}.$

\item The word maps may be composed to provide a word map for $\Gamma
_{C}<\Gamma_{A}.$

\item The word map may be used with the Todd-Coxeter algorithm to provide the
monodromy group $M(\Gamma_{A},\Gamma_{C})=M(A,C).$

\item The stabilizer of a point in $M(A,C)$ is $C/$\textrm{Core}$_{A}(C)$. If
\textrm{Core}$_{A}(C)$ is trivial then $C/$\textrm{Core}$_{A}(C)$ can be
tested to see if it is cyclic. The trivial core condition is satisfied in the
weakly malnormal case discussed in the next section.
\end{itemize}

\subsection{Constrained and tight pairs}

We need a mechanism to deal with families of inclusions. First we consider an
example arising from Singerman's list.

\begin{example}
\label{ex-parametric}Let $T(l,m,n)$ denote the triangle Fuchsian group with
signature $(l,m,n)$. Consider the possible case $\Gamma_{N}=T(2,d,2d)$ and
$\Gamma_{A}=T(2,3,2d)$ with $d\geq4.$ The index is
\[
\lbrack\Gamma_{A}:\Gamma_{N}]=\frac{1-\frac{1}{2}-\frac{1}{d}-\frac{1}{2d}%
}{1-\frac{1}{2}-\frac{1}{3}-\frac{1}{2d}}=\frac{(d-3)/2d}{(d-3)/6d}=3.
\]
With little more work a monodromy vector can be found $((1,2),(1,2,3),(1,3)),
$ and $M(\Gamma_{A},\Gamma_{N})=\ \Sigma_{3}.$ Notice that in this case
\[
o(\zeta_{1})=o(\pi_{1}),\text{ }o(\zeta_{2})=o(\pi_{2}),\text{ }o(\zeta
_{3})>o(\pi_{3}).
\]
and that $o(\zeta_{3})$ has a parameter $d\geq4.$
\end{example}

To handle the notion of families we extend our consideration of Fuchsian
groups to include parabolic elements $\delta_{1},\ldots,\delta_{q}.$ Thus we
have $\Delta$ $=$ $\langle\gamma_{1},\dots,\gamma_{t},$ $\delta_{1}%
,\ldots,\delta_{q}\rangle,$ with relations
\begin{equation}
\gamma_{1}^{n_{1}}=\gamma_{2}^{n_{2}}=\dots=\gamma_{t}^{n_{t}}=\prod
\limits_{i=1}^{t}\gamma_{i}\prod\limits_{j=1}^{q}\delta_{j}%
=1.\label{eq-relsparemetric}%
\end{equation}
The Teichm\"{u}ller dimension of the modified $\Delta$ is $d(\Delta)=t+q-3$.

\begin{definition}
Let $\rho:\Delta\rightarrow\Sigma_{n}$ as previously defined.

\begin{itemize}
\item A pair $\Gamma<\Delta$ is called \emph{constrained} if $\Delta$ has no
parabolic generators and $o(\zeta)=o(\rho(\zeta))$ for each elliptic generator.

\item A pair $\Gamma<\Delta$ is called \emph{tight} if $\Delta$ has at least
one parabolic generator and $o(\zeta)=o(\rho(\zeta))$ for each elliptic generator.
\end{itemize}
\end{definition}

\begin{remark}
The definition depends only on the cycle types and not the permutations
themselves. Hence, the definition depends only on the signature pair and may
be applied to a numerical projection $P:\mathcal{S}(\Gamma)\rightarrow
\mathcal{S}(\Delta)$.
\end{remark}

\begin{proposition}
Let $\Gamma<\Delta$ be a tight pair where $\Delta$ has $q$ parabolic elements.
Then there is a $q$ parameter family $\Gamma(\ell_{1},\ldots,\ell_{q}%
)<\Delta(\ell_{1},\ldots,\ell_{q})$ such that each member of the family has

\begin{itemize}
\item the same codimension $d(\Gamma,\Delta)$

\item the same index $[\Delta:\Gamma]$

\item the same monodromy $M(\Delta,\Gamma)$

\item the same word map
\end{itemize}
\end{proposition}

\begin{remark}
Every Fuchsian group pair is constrained or belongs to a unique family as
above. The tight pair defining the family is called the \emph{parent tight
pair}.
\end{remark}

\begin{example}
The triangle group family $T(2,d,2d)<T(2,3,2d)$ comes from the tight pair
$T(2,\infty,\infty)<T(2,3,\infty)$. The monodromy vector is $((1,2),$
$(1,2,3),$ $(1,3))$.
\end{example}

\subsection{Classification steps for pairs}

Here are steps for classification of the pairs $\Gamma_{N}<\Gamma_{A}%
.$\newline

\noindent\textit{Classify numerical projections by codimension.} For each
codimension there are a finite number of constrained pairs and a finite number
of tight pairs of numerical projections of signatures. The list of
codimensions will depend on how the signature pairs have been limited.\newline

\noindent\textit{Compute monodromy vectors.} For each candidate signature
pair, compute all the compatible monodromy vectors up to conjugacy
equivalence. Each constrained numerical projection gives rise to a finite
number (possibly none) of pairs $\Gamma_{N}<\Gamma_{A}$. Each tight numerical
projection gives rise to a finite number (possibly none) of parametric family
of pairs $\Gamma(\ell_{1},\ldots,\ell_{q})<\Delta(\ell_{1},\ldots,\ell_{q})$
all with the same monodromy. First one considers primitive pairs where
$M(\Delta,\Gamma)$ is a primitive permutation group. This can be done by
computer calculation and classification of primitive permutation groups (use
Magma or GAP). In the general case there is a tower $\Gamma_{N}=\Gamma
_{1}<\cdots<\Gamma_{e}=\Gamma_{A}$ such that each inclusion $\Gamma_{i}%
<\Gamma_{i+1}$ is a primitive pair, already classified. A tower may be fused
together by using word maps and the Todd Coxeter algorithm. An example of a
tower is $T_{7,7,7}<T_{3,3,7}<T_{2,3,7},$ which occurs for the $7$-gonal Klein quartic.

\subsection{Maximal actions and signatures}

Given a known group $G$ of automorphisms of a surface $S$, we want to know if
$G=A,$ i.e., $G$ has a maximal action. To demonstrate that $G$ already has a
maximal action in our examples in Section \ref{sec-samp}, we will use a simple
test on the signatures. Our test rests on the concept of finite maximality
developed in \cite{Sing2}. A Fuchsian group $\Gamma$ is called finitely
maximal if $\Gamma$ is not contained in any other Fuchsian group with finite
index. In \cite{Sing2} Singerman determines which Fuchsian groups are finitely maximal.

Now suppose that $G$ acts on $S,$ then we have
\begin{equation}%
\begin{array}
[c]{ccccc}%
\Pi & \hookrightarrow & \Gamma_{G} & \hookrightarrow & \Gamma_{A}\\
\downarrow\mathcal{\eta} &  & \downarrow\mathcal{\eta} &  & \downarrow
\mathcal{\eta}\\
\left\langle 1\right\rangle  & \hookrightarrow & G & \hookrightarrow & A
\end{array}
\label{eq-full}%
\end{equation}
If $\Gamma_{G}$ is finitely maximal then $\Gamma_{G}=\Gamma_{A}.$ If
$\Gamma_{G}$ is not finitely maximal we have
\[
\frac{\left\vert A\right\vert }{\left\vert G\right\vert }=\left\vert
\Gamma_{A}/\Gamma_{G}\right\vert =\frac{A(\Gamma_{G})}{A(\Gamma_{A})}%
=\frac{\mu(\Gamma_{G})}{\mu(\Gamma_{A})}%
\]
where $\left\vert A\right\vert /\left\vert G\right\vert $ is an integer
$k\geq2.$ If the signature of $G$ is $(m_{1},m_{2},\dots,m_{r})$ and the
signature of $A$ is $(n_{1},n_{2},\dots,n_{t})$ then this may be rewritten.
\begin{equation}
k=\frac{-2+%
{\displaystyle\sum\limits_{i=1}^{r}}
\left(  1-\frac{1}{m_{i}}\right)  }{-2+%
{\displaystyle\sum\limits_{j=1}^{t}}
\left(  1-\frac{1}{n_{j}}\right)  }=\frac{r-2-%
{\displaystyle\sum\limits_{i=1}^{r}}
\frac{1}{m_{i}}}{t-2+%
{\displaystyle\sum\limits_{j=1}^{t}}
\frac{1}{n_{j}}}\label{eq-maximal}%
\end{equation}
where $3\leq t\leq r$ and $m_{i}$ divides some $n_{j}$ for every $i.$ Equation
\ref{eq-maximal} provides a restriction which may be enough to prove finite
maximality since the quotient on the right hand side must be an integer.
Rather than state and prove a general result we give an example sufficient for
our needs. The example also follows from examining Singerman's list in
\cite{Sing2}.

\begin{example}
\label{ex-23mmax}The Fuchsian group with signature $(2,3,m),$ $m\geq7$ is
finitely maximal. To prove this let $h=t-2+%
{\displaystyle\sum\limits_{j=1}^{t}}
\frac{1}{n_{j}}$ (note that $t=3).$ Then
\[
1-\frac{1}{2}-\frac{1}{3}-\frac{1}{m}=kh,\text{ or }h=\frac{1}{6k}\frac
{m-6}{m}\text{ or }m=\frac{6}{6kh-1}%
\]
Since $k\geq2$, then $h<\frac{1}{12},$ and there are only a finite number of
signatures for which $h<\frac{1}{12}$ namely, $(2,3,7),$ $(2,3,8),$ $(2,3,9),$
$(2,3,10),$ $(2,3,11),$ and $(2,4,5).$ None of these yield an integer value
for $m$ for any integer value of $k.$
\end{example}

\section{Constraints on signatures\label{sec-const}}

As alluded to in Section $1$ the full classification problem of cyclic
$n$-gonal surfaces and their automorphism groups is too complex to be
completed in its entirety. In this section we discuss some methods to limit
the possible signature pairs $\mathcal{S}(\Gamma_{A})$ so that the problem is
more tractable. The limitations are chosen because of links to interesting
group theoretic, geometrical or arithmetic properties of the restrictions. As
discussed in the introduction there are two constraints we can consider.

\begin{itemize}
\item The action of $C$ on $S$ is weakly malnormal.

\item $S$ is a quasi-platonic surface.
\end{itemize}

\noindent The constraints in force because of weak normality have been
completely described and proven in \cite{Brou-Woo2}. Later we recall the main
theorem on signatures in that work, Theorem \ref{thm-sigsweak} below, and
indicate how the theorem may be proven by consideration of the action of $A$
on the singular $A$-orbits on $S.$ The constraints due to $S$ being a
quasi-platonic surface are not well known at this time other than to mention
that the signature of $A$ is quite restricted, and the potential for
application to dessins.

To understand the simplification offered by weakly malnormal actions we first
have to understand strong branching. Following that, we analyze of the action
of $A$ on the points lying over the branch points of $S\rightarrow S/A.$ This
analysis can be used to prove the signature theorem for weakly malnormal actions.

\subsection{Strong branching and weak normality $C<A$}

Previous work \cite{Acc}, \cite{Kont}, \cite{Sh} \cite{Woo1}, and \cite{Woo2}
has shown that if the $n$-gonal morphism $\pi_{C}:S\rightarrow S/C$ is highly
ramified then we often have $C\trianglelefteq A.$ This greatly simplifies the
calculation of $A$ since the calculations in Section \ref{sec-bot} may be
skipped. Some papers are restricted to the normal case \cite{Acc},
\cite{Kont}, \cite{Sh} \cite{Woo1}. The non-normal case has been considered in
\cite{Sh} and \cite{Woo2}. In \cite{Acc} Accola introduced a precise measure
of \textquotedblleft highly ramified\textquotedblright\ called strong
branching. Strong branching is a condition that guarantees normality in many
cases, in particular the prime cyclic case, and was used in \cite{Kont}, and
\cite{Woo2}. Strong branching may be used to conclude that is the genus of a
cyclic $n$-gonal surface is sufficiently large then $C\trianglelefteq A.$

An unramified covering $\pi:S_{1}\rightarrow S_{2}$ of degree $n$ satisfies
$2\sigma_{1}-2=n\left(  2\sigma_{2}-2\right)  .$ If the covering is ramified
then the formula is modified to:
\[
2\sigma_{1}-2=n\left(  2\sigma_{2}-2\right)  +R_{\pi}%
\]
where $R_{\pi}$ may be determined from the Riemann-Hurwitz formula. Accola
\cite{Acc} calls $\pi$ a strongly branched cover if
\[
R_{\pi}>2n(n-1)(\sigma_{2}+1)
\]
or
\[
\sigma_{1}>n^{2}\sigma_{2}+(n-1)^{2}.
\]
If $S_{2}$ has genus $0$ then the formulas are
\[
R_{\pi}>2n(n-1)
\]
or
\[
\sigma_{1}>(n-1)^{2}.
\]
In the case at hand, $\pi:S\rightarrow S/C$ given by equation (\ref{eq-basic}%
), if we define
\[
d_{i}=(n,p_{i}),\text{ }n_{i}=\frac{n}{d_{i}}%
\]
then
\[
R_{\pi}=n\sum\limits_{i=1}^{t}\left(  1-\frac{1}{n_{i}}\right)  =\sum
\limits_{i=1}^{t}\left(  n-d_{i}\right)  .
\]
We see, that $S\rightarrow S/C$ is strongly branched if, roughly, the right
hand side of equation (\ref{eq-basic}) has many factors. The main fact we need
about strong branching is the following.

\begin{proposition}
\label{pr-strong} Let $H$ be a group of automorphisms acting on a surface $S$
such that $S\rightarrow S/H$ is strongly branched. Then there is a unique
minimal, normal, nontrivial subgroup $L$ of $\mathrm{Aut(}{S)}$ such that
$L\leq H$, and $S\rightarrow S/L$ is strongly branched.
\end{proposition}

In \cite{Brou-Woo1} the concept of weak normality was introduced, to take
advantage of strong branching. It appears to be the weakest group theoretic
constraint such that we can take advantage of strong branching.

\begin{definition}
Let $H\leq G$ be a pair of groups and let $N=N_{G}(H)$. Then $H$ is weakly
malnormal in $G$ if for each $g\in G-N$ we have a trivial intersection $H\cap
H^{g}=\left\langle 1\right\rangle $. A group action of $H$ on a surface $S$ is
called weakly malnormal if $H$ is a weakly malnormal subgroup of
$A=\mathrm{Aut}(S)$.
\end{definition}

\begin{remark}
We can make some immediate remarks.

\begin{itemize}
\item Normal subgroups are trivially weakly malnormal.

\item If $H\leq G$ is a cyclic subgroup of prime order then $H$ is weakly
malnormal in $G$.

\item If $C\leq A$ is a cyclic subgroup of $A=Aut(S)$ and the map
$S\rightarrow S/C$ is fully ramified, then $C$ is weakly malnormal in $A$.

\item Let $H\leq G$ be a pair of groups such that $H$ is weakly malnormal in
$G$, but not normal. If $K$ is a nontrivial subgroup of $H$, then
$N_{G}(H)=N_{G}(K).$

\item Assume the same hypotheses as above. Then the representation of $G$ on
the left or right cosets of $H$ is faithful, for the kernel of the
representation is $\bigcap_{g\in G}H^{g}.$
\end{itemize}
\end{remark}

The main use of weak normality is given in the following proposition, which
shows that the non-normal cases occur only for small genus. For instance the
hyperelliptic involution is always normal and any non-normal cyclic trigonal
case must occur in genus $2,3$, or $4$.

\begin{proposition}
Let $H$ be a group of automorphisms acting on a surface $S$ such that
$S\rightarrow S/H$ is strongly branched and $H$ is weakly malnormal in
$A=\mathrm{Aut}(S)$. Then $H$ is normal in $A$. If the action of a group $C$
of order $n$ on a surface of genus $\sigma>(n-1)^{2}$ is weakly malnormal and
$S/C$ has genus zero then $C$ is normal in $A$.
\end{proposition}

\begin{example}
There are examples of cyclic $4$-gonal actions on surfaces of arbitrarily high
genus, but where $C$ is not normal in $A.$ See \cite{Brou-Woo2}.
\end{example}

The main restriction imposed for weakly malnormal actions is the signature
theorem below. The theorem is proved in \cite{Brou-Woo2} by directly working
with canonical generators, though it may also be proven from the analysis in
the next subsection.

\begin{theorem}
\label{thm-sigsweak} If the action of $C$ on $S$ is weakly malnormal, then
$\Gamma_{N}$ has at most $3$ additional periods to $\Gamma_{A}$. The
signatures for $\Gamma_{A}$ and $\Gamma_{N}$ appear as a pair in Table 5,
where $(a_{1},a_{2},a_{3})$ or $(k,k)$ is the signature of $K=\Gamma
_{N}/\Gamma_{C}$. The column labeled Codim is the Teichm\"{u}ller codimension
$d(\Gamma_{A},\Gamma_{N}).$
\end{theorem}%

\[%
\begin{tabular}
[c]{c}%
\textbf{Table 5}\\
$%
\begin{tabular}
[c]{||l|l|l|l||}\hline\hline
Case & Codim & Signature of $\Gamma_{N}$ & Signature of $\Gamma_{A}$\\\hline
$0A$ & $0$ & $(a_{1}m_{1},a_{2}m_{2},a_{3}m_{3},n_{1},\dots,n_{r})$ &
$(b_{1},b_{2},b_{3},n_{1},\dots,n_{r})$\\\hline
$0B$ & $0$ & $(km_{1},km_{2},n_{1},\dots,n_{r})$ & $(b_{1},b_{2},n_{1}%
,\dots,n_{r})$\\\hline
$1A$ & $1$ & $(a_{1}m_{1},a_{2}m_{2},a_{3}m_{3},n_{1},\dots,n_{r})$ &
$(b_{1},b_{2},n_{1},\dots,n_{r})$\\\hline
$1B$ & $1$ & $(km_{1},km_{2},n_{1},\dots,n_{r})$ & $(b_{1},n_{1},\dots,n_{r}%
)$\\\hline
$2A$ & $2$ & $(a_{1}m_{1},a_{2}m_{2},a_{3}m_{3},n_{1},\dots,n_{r})$ &
$(b_{1},n_{1},\dots,n_{r})$\\\hline
$2B$ & $2$ & $(km_{1},km_{2},n_{1},\dots,n_{r})$ & $(n_{1},\dots,n_{r}%
)$\\\hline
$3A$ & $3$ & $(a_{1}m_{1},a_{2}m_{2},a_{3}m_{3},n_{1},\dots,n_{r})$ &
$(n_{1},\dots,n_{r})$\\\hline\hline
\end{tabular}
\ \ $%
\end{tabular}
\ \ \ \
\]
\newline

\subsection{Orbits and induced generators}

We now return to the general situation. We want to closely link the signatures
of $\Gamma_{A}$ and $\Gamma_{N},$ by studying the singular $A$-orbits on $S$.

\begin{definition}
Let $H\subseteq A$ be any subgroup. The orbit $Hx$ is called $H$-regular if
$\left\vert Hx\right\vert =\left\vert H\right\vert ,$ and is called
$H$-singular if $\left\vert Hx\right\vert <\left\vert H\right\vert .$
\end{definition}

The facts in the following lemma are easily shown, we leave most details to
the reader.

\begin{lemma}
\label{lm-orbits}Let $H\subseteq M\subseteq A$ be subgroups of $A$. Then

\begin{enumerate}
\item If $Hx$ is singular then the order of the $H$-stabilizer $H_{y}$ of any
point $y\in Hx$ is $\left\vert H\right\vert /\left\vert Hx\right\vert .$

\item The orbit $Mx$ is a union of $H$-orbits, and $Mx$ is $M$-singular if any
of the $H$-orbits is $H$-singular.

\item Let $H\vartriangleleft M,$ $L=M/H,$ and $y$ be any point of $S$. Then
$My$ is a union of $H$-orbits of the same size. The number of $H$-orbits in an
$M$-orbit $My$ is less than $\left\vert L\right\vert $ if and only if $y$ is
fixed by an element of $M-H.$
\end{enumerate}
\end{lemma}

\begin{proof}
Only Statement 3 requires any work. Let $My$ be any $M$-orbit, it is a
disjoint union of $H$-orbits. Since $M$ normalizes $H$ then $M$ permutes the
$H$-orbits so they must all be the same size. Again by normality, $L$ permutes
the $H$-orbits comprising $My$ transitively. If there are less than
$\left\vert L\right\vert $ $H$-orbits, then for some $g,g_{1}\in M,$ $g\in
M-H,$ $gHg_{1}y=Hg_{1}y.$ It follows that $gh_{1}g_{1}y=g_{1}y$ for some
$h_{1}\in H$ and so $g_{1}^{-1}gh_{1}g_{1}y=y.$ If $g_{1}^{-1}gh_{1}%
g_{1}=g_{1}^{-1}gg_{1}g_{1}^{-1}h_{1}g_{1}\in H$ then so must $g_{1}%
^{-1}gg_{1}\in H$ and hence $g\in H.$ This is a contradiction and so
$g_{1}^{-1}gh_{1}g_{1}\in M-H.$ On the other hand if $M_{y}$ is not contained
in $H$ then $L$ has a nontrivial fixed point when acting on the set of
$H$-orbits. It follows that there are fewer than $\left\vert L\right\vert $
$H$-orbits.
\end{proof}

Now suppose that $\zeta\in\Gamma_{A}$ is a canonical generator. The elliptic
element $\zeta$ has a unique fixed point $z$ $\in\mathbb{H},$ let $x=\pi_{\Pi
}(z)$ be the image on $S.$ The map $\eta:\Gamma_{A}\rightarrow A$ maps
$\left\langle \zeta\right\rangle $ isomorphically onto the stabilizer $A_{x}.$
If a conjugate of $\zeta$ is chosen we simply get another point of $Ax.$ Thus,
there is a 1-1 correspondence between the canonical generators of $\Gamma_{A}$
and the singular orbits of $A$
\[
\zeta\leftrightarrow\left\langle \zeta\right\rangle ^{\Gamma_{A}%
}\leftrightarrow Ax.
\]
Moreover,
\begin{equation}
\left\vert \left\langle \zeta\right\rangle \right\vert =\left\vert
A_{x}\right\vert =\left\vert A\right\vert /\left\vert Ax\right\vert
.\label{eq-order}%
\end{equation}
A similar statement applies to any subgroup $H$ $\subseteq A$ and the elliptic
canonical generators of the corresponding group $\Gamma_{H}.$ The following
proposition details the relationship between induced generators and singular orbits.

\begin{proposition}
\label{pr-orbits_gens}Let $H\subseteq M\subseteq A$ be a tower of groups and
$\Gamma_{H}\subseteq\Gamma_{M}\subseteq\Gamma_{A}$ the covering Fuchsian
groups. Assume that the genus of $S/H$ is zero. Then, we have the following.

\begin{enumerate}
\item The canonical generators $\zeta\in\Gamma_{H}$ are in 1-1 correspondence
with the $H$-singular orbits $Hx$ via%
\[
\zeta\leftrightarrow\left\langle \zeta\right\rangle ^{\Gamma_{H}%
}\leftrightarrow Hx,
\]
where $x=\pi_{\Pi}(z)$ for the fixed point $z$ of $\zeta,$ $\left\langle
\zeta\right\rangle ^{\Gamma_{H}}$ is a conjugacy class of stabilizers in
$\Gamma_{H}.$ The order of $\zeta$ is $\left\vert H_{x}\right\vert =\left\vert
H\right\vert /\left\vert Hx\right\vert .$

\item Suppose that $\zeta\in\Gamma_{A}$ is a canonical generator and $z,x$ are
as item 1$.$ Then the canonical generators of $H$ induced by $\zeta$ are in
1-1 correspondence to the singular $H$-orbits contained in $Ax.$ Moreover if
$Hy\subseteq Ax$ is a singular $H$-orbit then the order of the corresponding
induced canonical generator of $\Gamma_{H}$ is $\left\vert H\right\vert
/\left\vert Hy\right\vert .$

\item Let $gens(\Gamma)$ be a set of canonical generators of $\Gamma.$ Then
the signatures satisfy
\[
\left\vert gens(\Gamma_{A})\right\vert \leq\left\vert gens(\Gamma
_{M})\right\vert \leq\left\vert gens(\Gamma_{H})\right\vert .
\]

\end{enumerate}
\end{proposition}

\begin{proof}
Statement 1 was demonstrated in the discussion preceding the statement of the
Proposition. Statement 2 is a reformulation of the theorem from Singerman. To
prove statement 3 observe that
\begin{align*}
\left\vert gens(\Gamma_{M})\right\vert  &  =3+\text{\textrm{Teichm\"{u}ller
dimension }}\Gamma_{M}\\
&  \leq3+\text{\textrm{Teichm\"{u}ller dimension }}\Gamma_{H}\\
&  =\left\vert gens(\Gamma_{H})\right\vert
\end{align*}
The identical argument works for the other inequality.
\end{proof}

We are now going to focus on the relation between the singular $N$-orbits and
the singular $A$-orbits when $S/C$ has genus zero. To this end we identify
exactly three ways in which an $N$-orbit can be singular.

\begin{remark}
\label{rk-1}Let the notation for the groups $C\subseteq N\subseteq A$ be as
above and assume that $K=N/C$ is non-trivial. Then the singular $N$-orbits are
of three types:

\begin{enumerate}
\item Type 1: The orbit $Nx$ consists of $\left\vert K\right\vert $ singular
$C$-orbits. For each $y\in$ $Nx$ the stabilizer $N_{y}\subseteq C$ and so
$N_{y}=C_{y}$. This is according to Statement 3 of Lemma \ref{lm-orbits}.

\item Type 2: The orbit $Nx$ consists of fewer than $\left\vert K\right\vert $
regular $C$-orbits. For each $y\in$ $Nx$ $N_{y}\cap C$ is trivial. There is an
element $g\in N$ of order $\left\vert N\right\vert /\left\vert Ny\right\vert $
in $N$ such that each stabilizer in $Nx$ is conjugate to $\left\langle
g\right\rangle .$ The order of $g$ is one of the periods of $K.$

\item Type 3: The orbit $Nx$ consists of fewer than $\left\vert K\right\vert $
singular $C$-orbits. Let $\overline{x}$ be the orbit $Nx,$ so that $\left\vert
K\overline{x}\right\vert <\left\vert K\right\vert .$ The value $a=\left\vert
K\right\vert /\left\vert K\overline{x}\right\vert $ is one of the periods of
$K$ acting on $S/H,$ let $m=\left\vert N_{x}\cap C\right\vert .$ Then there is
an element $g\in N$ of order $am$ such that $\left\langle g\right\rangle
=N_{x},$ and $\left\langle g^{a}\right\rangle =N_{x}\cap C.$
\end{enumerate}
\end{remark}

\begin{remark}
(Continuation of above Remark) Suppose that $K$ is non-trivial. Then there are
two possible signatures of $\Gamma_{N},$ namely $(km_{1},km_{2},$
$n_{1},\ldots,n_{r})$ or $(a_{1}m_{1},a_{2}m_{2},a_{3}m_{3},n_{1},\ldots
,n_{r})$, depending on the signature of $K$. The orbits of Type 1 produce the
canonical generators of orders $n_{1},\ldots,n_{r}.$ The orbits of Type 2 and
3 produce canonical generators of orders $a_{1}m_{1},a_{2}m_{2},a_{3}m_{3}$ or
$km_{1},km_{2}$ depending on signature of $K.$ For the orbits of Type 2,
$m_{i}=1.$ There are either two or three orbits of Type 2 or 3 if $K$ is
non-trivial. If $K$ is trivial then the only singular orbits are of Type 1.
$\ $The canonical generators corresponding to the orbits of type 2 or 3 are
$K$-generators.
\end{remark}

\begin{remark}
In the previous situation we do not need $C$ to be cyclic.
\end{remark}

Now let us assume that $C$ is weakly malnormal in $A$ and determine the
consequences for the orbits and the signatures.

\begin{lemma}
\label{lm-orbsweak1}Suppose that $C$ is weakly malnormal in $A$. Then we have
the following:

\begin{enumerate}
\item If $Ny\subseteq Ax$ is a singular orbit of Type 1 or Type 3 we have
equality of stabilizers $N_{y}=A_{y}.$

\item Each singular orbit $Ax$ contains at most one $N$-orbit of Type 1.
\end{enumerate}
\end{lemma}

\begin{proof}
Assume that $Ny\subseteq Ax$ is a singular orbit of Type 1 or Type 3. By
definition the stabilizer $C_{y}$ is nontrivial and the cyclic subgroup
$A_{y}\supseteq C_{y}$ and so $A_{y}$ normalizes a non-trivial subgroup of $C.
$ It follows that $A_{y}\subseteq N$ and hence $N_{y}=A_{y}.$

Suppose that $Ax$ is a singular $A$-orbit and that $Ax$ contains an $N$-orbit
$Ny$ of Type 1. Then, by Remark \ref{rk-1} and the first statement above,
$C_{y}=N_{y}=A_{y}$. Now suppose that $Ngy$ is a Type 1 orbit distinct from
$Ny$ for some $g$ $\in A-N.$ We must also have that
\[
C_{gy}=A_{gy}=gA_{y}g^{-1}=gC_{y}g^{-1}.
\]
But then
\[
C_{gy}\subseteq C\cap gC_{y}g^{-1}\subseteq C\cap gCg^{-1}=\left\langle
1\right\rangle .
\]
Thus we have a contradiction if there are two distinct $N$-orbits of Type 1.
\end{proof}

\begin{remark}
The proof techniques just used automatically shows the following. If the
$n$-gonal morphism $S\rightarrow S/C$ is fully ramified, i.e., has signature
$(n,\ldots,n),$ then the action is weakly malnormal. To see this let $x$ be
any point of $S$ and observe that the stabilizer $C_{x}=C$ or $C_{x}%
=\left\langle 1\right\rangle .$ There is some point $x\in S$ where
$C=C_{x}\subseteq A_{x}.$ Let $g\in A-N$ and consider
\begin{align*}
C_{gx}  &  =\left\{  c\in C:cgx=gx\right\} \\
&  =\left\{  c\in C:\left(  g^{-1}cg\right)  x=x\right\} \\
&  =C\cap gA_{x}g^{-1}\supseteq C\cap gCg^{-1}.
\end{align*}
If $\left\vert C_{gx}\right\vert =1$ then $C\cap gCg^{-1}=\left\langle
1\right\rangle .$ On the other hand, seeking a contradiction, assume that
$\left\vert C_{gx}\right\vert >1.$ By the fully ramified hypothesis we must
have $\left\vert C_{gx}\right\vert =\left\vert C\right\vert ,$ but then, we
must have $C_{gx}=C$, which implies $gA_{x}g^{-1}\supseteq C.$ As $A_{x}$ and
$gA_{x}g^{-1}$ are cyclic they have unique subgroups of order $\left\vert
C\right\vert $ and we conclude that $C=gCg^{-1},$ a contradiction.
\end{remark}

\begin{remark}
Lemma \ref{lm-orbsweak1} still holds if $C$ is not cyclic.
\end{remark}

It is useful to classify the decomposition of $A$-orbits into $N$-orbits.

\begin{lemma}
\label{lm-orbsweak2}Assume that $C$ is weakly malnormal in $A,$ and assume
that $\zeta$ is a canonical generator of $\Gamma_{A}$ which corresponds to the
orbit $Ax.$ Then we have the following possibilities.

\begin{enumerate}
\item The orbit $Ax$ contains no singular $N$-orbits and $\zeta$ does not
induce any canonical generator of $\Gamma_{N}$ or $\Gamma_{C}.$

\item The orbit $Ax$ contains a singular $N$-orbit of Type 1 and no other
singular orbits. Then $\zeta$ induces a canonical generator $\theta$ of
$\Gamma_{N}$ of the same order as $\zeta,$ and exactly $\left\vert
K\right\vert $ canonical generators of $\Gamma_{C}$ with same order as
$\zeta.$

\item The orbit $Ax$ is as in 1, 2 above except that it additionally contains
up to three $N$-orbits of Type 2 or type 3 subject to the constraint that the
total number of orbits of Type 2 and Type 3 is $2$ or $3$. Let $\theta$
$\in\Gamma_{N}$ be the corresponding generator induced by $\zeta$ for an orbit
of Type 2 or 3. In the case of Type 2 we have $o(\theta)<o(\zeta)$ and in the
type 3 case we have $o(\theta)=o(\zeta).$
\end{enumerate}
\end{lemma}

\begin{proof}
In Case 1 all the $N$-orbits and $C$-orbits are regular and hence no canonical
generators are induced. The number and order of induced canonical generators
in the remaining cases follow from Lemma \ref{lm-orbits} and Lemma
\ref{lm-orbsweak1}.
\end{proof}

The discussion in the section may be used to prove the signature theorem
\ref{thm-sigsweak}. Here is a proof sketch. Consider any canonical generator
$\zeta$ of $\Gamma_{A},$ the corresponding singular $A$-orbit $Ax$ and its
decomposition into $N$-orbits. If $Ax$ contains an $N$-orbit of Type 1, then
$\zeta$ induces one canonical generator of $\Gamma_{N}$ of the same order as
$\zeta$ and possibly others. Since two Type 1 $N$-orbits cannot occupy the
same $A$-orbit then all the canonical generators of $\Gamma_{A}$ inducing Type
1 generators of $\Gamma_{N}$ are distinct. This leads to the sequence
$n_{1},\dots,n_{r}$ in both signatures. It follows that%
\[
\left\vert gens(\Gamma_{N})\right\vert \geq\left\vert gens(\Gamma
_{A})\right\vert \geq\left\vert gens(\Gamma_{N})\right\vert -3.
\]
The remaining generators come from Type 2 and Type 3 orbits. Thus the
signature $\Gamma_{N}$ is known and the periods $b_{i}$ of $\Gamma_{A}$ are
simply fill-ins, except that each period of $\Gamma_{N}$ must divide some
period of $\Gamma_{A}$.

\begin{remark}
If $K$ is trivial then by the argument in the proof of the proposition of the
proposition $\Gamma_{A}$ and $\Gamma_{N}$ have the signature $(n_{1}%
,\dots,n_{r}).$ It follows that $A=N$ and $C$ is normal in $A.$
\end{remark}

\begin{example}
Consider Klein's quartic. $C=\mathbb{Z}_{7},$ $C=\mathbb{Z}_{3}\ltimes
\mathbb{Z}_{7},$ $A=PSL_{2}(7),$ $K=\mathbb{Z}_{3}.$ Then we have%
\[%
\begin{tabular}
[c]{c}%
\textbf{Table 6}\\
$%
\begin{tabular}
[c]{|l|l|l|}\hline
Signature of $\Gamma_{A}$ & Signature of $\Gamma_{N}$ & Signature of
$\Gamma_{C}$\\\hline
$(2,3,7)$ & $(3,3,7)$ & $(7,7,7)$\\\hline
\end{tabular}
\ \ \ $%
\end{tabular}
\ \ \
\]
The $A$-orbits split into $N$-orbits as follows
\[%
\begin{tabular}
[c]{c}%
\textbf{Table 7}\\
$%
\begin{tabular}
[c]{|l|l|l|l|l|}\hline
order of canonical generator & $2$ & $3$ & $7$ & Size of $N$-orbits\\\hline
Size of $A$-orbit & $84$ & $56$ & $24$ & \\\hline
Regular $N$-orbits & $4$ & $2$ & $1$ & $21$\\\hline
Type 1 $N$-orbits & $0$ & $0$ & $1$ & $3$\\\hline
Type 2 $N$-orbits & $0$ & $2$ & $0$ & $7$\\\hline
Type 3 $N$-orbits & $0$ & $0$ & $0$ & $1$\\\hline
\end{tabular}
\ \ \ $%
\end{tabular}
\ \ \ \ \ \ \ \ \
\]
The table entries are interpreted as follows The $A$-orbit corresponding to a
canonical generator of order 2 consists of $84$ points which breaks up into
$4$ regular $N$-orbits of size $21.$ The $A$-orbit corresponding to a
canonical generator of order 3 has $56$ points and breaks up into $2$ regular
$N$-orbits of size $21$ and $2$ Type $2$ orbits of size $7$. The $A$-orbit
corresponding to a canonical generator of order 7 has $246$ points and breaks
up into a regular $N$-orbit of size $21$ and one Type $1$ orbit of size $3.$
There cannot be any Type 3 orbits since $N$ is not cyclic.
\end{example}

\section{Examples\label{sec-samp}}

\subsection{Constrained examples}

Only two constrained examples have been found as of the writing of this paper.
Both are discussed in \textit{\cite{Woo2} }and are well known curves\textit{.}%

\[%
\begin{tabular}
[c]{|l|l|l|l|l|l|}\hline
Name & genus & $C$ & $N$ & $A$ & $K$\\\hline
Klein's quartic & 3 & $C_{7}$ & $C_{3}\ltimes C_{7}$ & $PSL_{2}(7)$ & $C_{3}%
$\\\hline
Bring's curve & 4 & $C_{5}$ & $C_{4}\ltimes C_{5}$ & $\Sigma_{5}$ & $C_{4}%
$\\\hline
\end{tabular}
\ \
\]

\subsection{Examples with parametric families.}

We conclude with some examples of parametric families suggested Table 4. First
we consider Fermat curves. We will show that they give a parametric family of
curves with weakly malnormal cyclic $n$-gonal actions where the automorphism
group strictly contains the normalizer of the $n$-gonal action.

\begin{example}
Consider the variant of the Fermat curve $F_{n}$ given by
\begin{equation}
x^{n}+y^{n}=-1,\label{eq-fermataffine}%
\end{equation}
or better, by its homogeneous equation,
\begin{equation}
X^{n}+Y^{n}+Z^{n}=0.\label{eq-fermatproj}%
\end{equation}
From the affine equation \ref{eq-fermataffine} we see that the curve is cyclic
$n$-gonal. From the projective form \ref{eq-fermatproj}, we determine that the
linear group $\Sigma_{3}\ltimes\mathbb{Z}_{n}^{3},$ acts on $F_{n} $ with
$\Sigma_{3}$ acting as permutations of the coordinates, and $\mathbb{Z}%
_{n}^{3}$ acting by
\[
(a,b,c)\cdot(X:Y:Z)=\ \left(  e^{2\pi ia/n}X:e^{2\pi ib/n}Y:e^{2\pi
ib/n}Z\right)  ,
\]
in homogeneous coordinates. The diagonal subgroup $D=\left\{  (a,a,a):a\in
\mathbb{Z}_{n}\right\}  $ acts trivially. In fact, $G=\Sigma_{3}%
\ltimes\mathbb{Z}_{n}^{3}/D$ is the automorphism group of $F_{n}$ as we shall
see shortly. The affine model of $F_{n}$ can be obtained via $x=X/Z$ and
$y=Y/Z.$ In the affine setting $(a,b,c)$ acts via
\[
(a,b,c)\cdot(x,y)=\left(  \exp\left(  2\pi i\frac{a-c}{n}\right)
x,\exp\left(  2\pi i\frac{b-c}{n}\right)  y\right)  \
\]
and the coordinate transpositions correspond to birational maps as in the
table following.
\[%
\begin{tabular}
[c]{|l|l|l|}\hline
permutation & projective automorphism & birational affine map\\\hline
$(1,2)$ & $(X:Y:Z)\leftrightarrow(Y:X:Z)$ & $(x,y)\rightarrow(y,x)$\\\hline
$(1,3)$ & $(X:Y:Z)\leftrightarrow(Z:Y:X)$ & $(x,y)\rightarrow(1/x,y/x)$%
\\\hline
$(2,3)$ & $(X:Y:Z)\leftrightarrow(X:Z:Y)$ & $(x,y)\rightarrow(x/y,1/y)$%
\\\hline
\end{tabular}
\]
Using the projection $\mathbb{Z}_{n}^{3}\rightarrow\mathbb{Z}_{n}^{2},$
$(a,b,c)\rightarrow(a-c,b-c)$ with kernel $D$, we may more conveniently denote
the automorphism group of the affine model by writing $G\backsimeq\Sigma
_{3}\ltimes\mathbb{Z}_{n}^{2}$ with $(1,2),(1,3),$ and $(2,3)$ acting on
$\mathbb{Z}_{n}^{2}$ by $(1,2):(a,b)\rightarrow(b,a),$ $(1,3):(a,b)\rightarrow
(-a,-a+b),$ and $(2,3):(a,b)\rightarrow(a-b,-b).$

Let $C=\mathbb{Z}_{n}$ be the cyclic group $\left\langle 1\right\rangle
\ltimes\left\{  (0,a,0)):a\in\mathbb{Z}_{n}\right\}  $ of $\Sigma_{3}%
\ltimes\mathbb{Z}_{n}^{3}$ corresponding to the subgroup $\left\langle
1\right\rangle \ltimes\left\{  (0,a):a\in\mathbb{Z}_{n}\right\}  $ of $G$. The
subgroup $C$ corresponds to the standard $n$-gonal action on $F_{n}$ when we
write the affine equation in the form $y^{n}=-1-x^{n},$ with projection
$(x,y)\rightarrow x,$ or $(X:Y:Z)\rightarrow(X:Z)$ in projective coordinates.
The normalizer $N$ of $C$ in $G$ is $N=\left\langle (1,3)\right\rangle
\ltimes\mathbb{Z}_{n}^{2}.$ The group $N$ acts by multiplying $X$ and $Z$ by
$n$'th roots of unity and switching the $X$ and $Z $ coordinates. There are
two other $n$-gonal projections, the information for all three projections is
summarized in the following table. The projections are equivalent under the
full automorphism group.
\[%
\begin{tabular}
[c]{|l|l|l|l|}\hline
affine map & projective map & $C$ & $N$\\\hline
$(x,y)\rightarrow x$ & $(X:Y:Z)\rightarrow(X:Z)$ & $\left\{  (0,a):a\in
\mathbb{Z}_{n}\right\}  $ & $\left\langle (2,3)\right\rangle \ltimes
\mathbb{Z}_{n}^{2}$\\\hline
$(x,y)\rightarrow y$ & $(X:Y:Z)\rightarrow(Y:Z)$ & $\left\{  (a,0):a\in
\mathbb{Z}_{n}\right\}  $ & $\left\langle (1,3)\right\rangle \ltimes
\mathbb{Z}_{n}^{2}$\\\hline
$(x,y)\rightarrow x/y$ & $(X:Y:Z)\rightarrow(X:Y)$ & $\left\{  (a,a):a\in
\mathbb{Z}_{n}\right\}  $ & $\left\langle (1,2)\right\rangle \ltimes
\mathbb{Z}_{n}^{2}$\\\hline
\end{tabular}
\]

By Example \ref{ex-ngonalsignature} the signature of $C$ is $(n,\ldots,n)$ (n
times) since $x^{n}+1$ has distinct linear factors. The dihedral action of
$N/C\backsimeq\mathbb{Z}_{2}\ltimes\mathbb{Z}_{n}\ $on $S/C$ is the standard
dihedral action of $D_{n}$ on the sphere with fixed points as follows. The
branch points of $C$ form a single $N/C$ orbit consisting of points of
ramification order 2, there is another $N/C$ orbit of points of ramification
order $2$ consisting of $C$-regular points, and finally an orbit consisting
two $C$-regular points with $N/C$ ramification order $n$. Therefore, the
signature of $N$ is $(2,2n,n)$ when written compatibly with the $(2,2,n)$
signature of $K=D_{n}.$ \ Alternatively, it is easily directly verified that
$N/C\backsimeq D_{n}.$ For completeness let us construct a generating
$(2,2n,n)$-vector for $N=\left\langle (1,2)\right\rangle \ltimes\mathbb{Z}%
_{n}^{2}$. Let $g=$ $((1,2),(0,0)),$ $h=(1,(1,0)),$ and $k=(1,(0,1)).$ Then
$g^{2}=1,h^{n}=k^{n}=1,$ and $\ ghgh=kh.$ Hence $gh$ has order $2n$ and
$(g,gh,h^{-1})$ is a generating $(2,2n,n)$-vector for $N$.

Now we turn to the full automorphism group. It is easily verified that $C$ is
weakly malnormal in $G$. If we can show $G$ is the full automorphism group
then our example is complete. The monodromy group $M(G,N)$ is easily
calculated to be $\Sigma_{3},$ and a little work shows that the signature of
$G$ is $(2,3,2n).$ By Example \ref{ex-23mmax}, $\Gamma_{G}$ is finitely
maximal and so the full automorphism group $A$ of $F_{n}$ equals $G.$ This
case is line 9 from Table 4 with $k=d=n$ and $e=1$. Again for completeness we
construct a generating $(2,3,2n)$-vector for $G$. The generating vector
projects to a $(1\cdot2,3,1\cdot2)$-monodromy vector of $M(G,N)$ and so we
start of with the elements $g_{1}=((1,3),(0,0))$ and $g_{2}=((1,2),(0,0))$ of
$\Sigma_{3}$ to construct $(g_{1},g_{1}g_{2},g_{2})$ a $(1\cdot2,3,1\cdot
2)$-monodromy vector. Setting $h=(1,(1,0))$ as before we see that $\left(
g_{1}g_{2}h^{-1}\right)  ^{3}=1$ and $o(hg_{2})=2n.$ Thus $(g_{1},g_{1}%
g_{2}h^{-1},hg_{2})$ is a generating $(2,3,2n)$-vector.
\end{example}

\paragraph{General $K=D_{k},$ with trivial action on $C$}

Before going on to our remaining examples we shall examine the general case
where $K=D_{k}$ and the $D_{k}$-action on $C$ is trivial and see what we can
conclude about the structure of $N.$ From Table 2 we see that $D_{k}$ acts on
$C$ by factoring through a group of order 2 or 4. Because of space
considerations we are going to restrict our attention to the case where all of
$D_{k}$ acts trivially on $C$. In the sequence $C\rightarrow N\rightarrow
D_{k},$ let $E$ be the inverse image of $C_{k}\subset D_{k},\ $so that
$E\vartriangleleft N$ and we have an exact sequence $C\rightarrow E\rightarrow
C_{k},$ with $C_{k}$ acting trivially on $C.$ In this case, it may be verified
simply that $E$ is an abelian group, which we shall write in additive
notation. For $x\in E$ we denote the image in $C_{k}$ by $\overline{x}.$

\begin{remark}
In case $k$ is odd then $C_{k}\subset D_{k}$ automatically acts trivially on
$C$ and $E$ is automatically abelian.
\end{remark}

The structure of $E$ is strongly influenced by the action of $D_{k}$ on $E$.
If $g\in N-E$ then conjugation by $g$ induces an automorphism $\phi$ of $E.$
Since $g^{2}\in C_{k},$ then $\phi^{2}=1$ and $\phi$ does not depend on the
$g$ chosen. The subgroup $C$ is invariant under $\phi,$ and by assumption
$\phi(x)=x,$ $x\in C.$ On the quotient group $C_{k}=E/C$ the induced map acts
by $\overline{\phi}(\overline{x})=-\overline{x}.$ To utilize the action of
$\phi$ to determine the structure $E$ we need the even/odd Sylow decomposition
of $E.$ Decompose $E=S_{2}\times S_{o}$ where $S_{2}$ is the 2-Sylow subgroup
of $E$ and $S_{o}$ is the subgroup of elements of odd order, a direct sum of
odd order Sylow subgroups. We will determine the structure of the two
subgroups separately, considering the odd piece first.

Since $S_{o}$ has odd order, division by $2$ is well-defined. We will use the
\textquotedblleft eigenspace\textquotedblright\ decomposition of $E$ induced
by $\phi.$ For any $x$ we may write
\begin{equation}
x=x^{+}+x^{-}\ \label{eq-eigen1}%
\end{equation}
where
\begin{equation}
x^{+}=\frac{x+\phi(x)}{2},\text{ }x^{-}=\frac{x-\phi(x)}{2}.\label{eq-eigen2}%
\end{equation}
We observe from equation \ref{eq-eigen2} and $\phi^{2}=1$ that $\phi
(x^{+})=x^{+}$ and $\phi(x^{-})=-x^{-}.$ Let $S_{o}^{+}=\left\{  x\in
S_{o}:\phi(x)=x\right\}  $ and $S_{o}^{-}=\left\{  x\in S_{o}:\phi
(x)=-x\right\}  .$ By equation \ref{eq-eigen1} $S_{o}=$ $S_{o}^{+}+S_{o}^{-}.$
We also have $S_{o}^{+}\cap S_{o}^{-}$ is trivial since any $x\in S_{o}%
^{+}\cap S_{o}^{-}$ satisfies $x=\phi(x)=-x,$ forcing $x=0,$ by the odd order
condition. It follows that $S_{o}=$ $S_{o}^{+}\oplus S_{o}^{-}.$ Now
$C\subseteq S_{o}^{+}.$ If $C$ is properly contained in $S_{o}^{+}$ then some
element $\overline{x}$ $\in$ $S_{o}^{+}/C\subseteq E/C$ satisfies
$\overline{\phi}(\overline{x})=\overline{x},$ a contradiction. It follows that
$C=S_{o}^{+}$ and $S_{o}^{-}$ maps injectively to $C_{k}.$ Both $S_{o}^{+}$
and $S_{o}^{-}$ are cyclic.

The analysis of $S_{2}$ is a bit more fussy. Let $S_{C}=S_{2}\cap C$ and
$\overline{S}=S_{2}/S_{C}$ be the corresponding $2$-Sylow subgroups of $C$ and
$C_{k}=E/C.$ Note that $S_{C}\rightarrow S_{2}\rightarrow\overline{S}\ $is
exact, and $S_{C}$ and $\overline{S}$ are cyclic and we may suppose they are
non-trivial. As in our previous analysis set $S_{2}^{+}=\left\{  x\in
S_{2}:\phi(x)=x\right\}  $ and $S_{2}^{-}=\left\{  x\in S_{2}:\phi
(x)=-x\right\}  .$ Now, however, $S_{2}^{+}\cap S_{2}^{-}=\left\{  x\in
S_{2}:x=-x\right\}  $ the subgroup of elements of order 2. Also it is not true
that $S_{2}=$ $S_{2}^{+}+S_{2}^{-}$. Let $h\in S_{2}$ be an element such that
$\overline{S}=$ $\left\langle \overline{h}\right\rangle ,$ set $H=\left\langle
h\right\rangle .$ Then the map \ $S_{C}\times H\rightarrow S_{2},$
$(x,y)\rightarrow xy$ is surjective and the kernel is $Z=S_{C}\cap H=\left\{
(x,-x):x\in S_{C}\cap\left\langle h\right\rangle \right\}  $ isomorphic to a
cyclic subgroup of $H$. We know that $\phi(h)=-h+c$ for some $c\in C.$ If we
can choose $c=0$ then the $S_{C}\cap H$ is contained in $S_{2}^{+}\cap
S_{2}^{-}$ and so $Z$ has order $2$ or $1$. This leads to two cases.
\begin{equation}
S_{2}\backsimeq C\times\left\langle h\right\rangle \label{eq-2sylow1}%
\end{equation}
and
\begin{equation}
S_{2}\backsimeq C\times\left\langle h\right\rangle /Z,\text{ }\left\vert
Z\right\vert =2.\label{eq-2sylow2}%
\end{equation}
As an example of the first take $N=S_{C}\times D_{k}$ where $k=\left\vert
\overline{S}\right\vert .$ $\ $ For the second let $E=S_{C}\times C_{2k}$ and
let $z_{1},$ $z_{2}$ be the unique elements of order 2 in $S_{c}$ and
$C_{2k},$ and let $Z=$ $\left\langle (z_{1},z_{2})\right\rangle .$ Then $Z$ is
a normal subgroup of $S_{C}\times D_{2k},$ and $S_{C}\times D_{2k}/Z$ is an
example satisfying equation \ref{eq-2sylow2}. In the first case $S_{2}$ is a
product of two cyclic groups. In the second case if $\left\vert S_{C}%
\right\vert =2$ or $\left\vert H\right\vert =2$ then $S_{2}$ is cyclic,
otherwise $S_{2}$ is the product of two cyclic 2-groups.

Next suppose that we cannot choose $c=0.$ If $b\in C$ is any other element and
$h^{\prime}=h+b$ then $\phi(h^{\prime})=-h+c+b=-h^{\prime}+c+2b$. Thus if
$c\notin2C$ then we may choose $c$ to be an explicit generator of $C $ as
$C/2C=C_{2}.$ Also observe that $h$ and $\phi(h)$ have the same order and so
$\left\vert C\right\vert \leq\left\vert H\right\vert .$ As before, construct
the exact sequence $Z\rightarrow S_{C}\times\left\langle h\right\rangle
\rightarrow S_{2}$ so that the kernel $Z=$ $\left\{  (x,-x):x\in S_{C}%
\cap\left\langle h\right\rangle \right\}  $ $=\left\langle rh\right\rangle $
for some $r,$ a power of $2$. Since the sequence $Z\rightarrow S_{C}%
\times\left\langle h\right\rangle \rightarrow S_{2}$ is short exact
\[
\left\vert Z\right\vert =\frac{\left\vert S_{C}\times\left\langle
h\right\rangle \right\vert }{\left\vert S_{2}\right\vert }=\frac{\left\vert
S_{C}\right\vert o(h)}{\left\vert S_{C}\right\vert \left\vert \overline
{S}\right\vert }=\frac{o(h)}{\left\vert \overline{S}\right\vert }%
\]
and $r=$ $o(h)/\left\vert Z\right\vert =\left\vert \overline{S}\right\vert
>1.$ As $rh\in C$ then $rh=sc$ for some $s.$ Noting that $rh$ is $\phi
$-invariant we get $rh=\phi(rh)=-rh+rc,$ or $rc=2rh.$ Since both $S_{C}$ and
$H$ are nontrivial it follows that $o(h)=2o(c).$ From $rc=2rh=2sc,$ there are
two possible values for $s,$ namely $s=\frac{r}{2}$ and $s=\frac{r+o(c)}{2}.$
Thus $Z=\left\langle \left(  \frac{r}{2}c,-rh\right)  \right\rangle $ or
$Z=\left\langle \left(  \left(  \frac{r+o(c)}{2}\right)  c,-rh\right)
\right\rangle .$ Correspondingly, assuming $r>1$ is a power of 2, we may
construct a model for the $\phi$-module $S_{2}$ as $S_{C}\times\left\langle
h\right\rangle /Z,$ where $Z=\left\langle \left(  \frac{r}{2}c,-rh\right)
\right\rangle $ or $Z=\left\langle \left(  \left(  \frac{r+o(c)}{2}\right)
c,-rh\right)  \right\rangle $ and $\phi:$ $S_{C}\times\left\langle
h\right\rangle \rightarrow$ is the map $(nc,mh)\rightarrow((n+m)c,-mh).$ The
subgroup $Z$ is invariant since $\phi\left(  \frac{r}{2}c,-rh\right)  =$
$\left(  \left(  \frac{r}{2}-r\right)  c,rh\right)  =$ $\left(  \frac{r}%
{2}c,-rh\right)  ,$ or $\phi\left(  \frac{r}{2}c,-rh\right)  =\left(  \left(
\frac{o(c)}{2}-\frac{r}{2}\right)  c,rh\right)  =-\left(  \left(
\frac{r+o(c)}{2}\right)  c,-rh\right)  .$ It follows that $S_{C}%
\times\left\langle h\right\rangle /Z$ has all the correct properties.

The preceding discussion gives us a good representation of $S_{2}$ when it is
not cyclic. It will be useful to write down an alternate, specific
representation of $S_{2}$ when it is cyclic. Let let $S_{2}=\left\langle
h\right\rangle ,$ $q=o(h),$ and let $\overline{h}$ be the image of in
$\overline{S},$ so that $S_{C}=\left\langle rh\right\rangle ,$ and
$\overline{S}=\left\langle \overline{h}\right\rangle .$ Then
\begin{align}
\phi(h)  & =ah,\label{eq-phi1}\\
\phi(rh)  & =rh,\text{ }\label{eq-phi2}\\
\overline{\phi}(\overline{h})  & =-\overline{h}.\label{eq-phi3}%
\end{align}
Since $\overline{h}+\overline{\phi}(\overline{h})=0$ then $\overline{h+ah}=0 $
and so $ah=-h+lrh$ for some $l$ or $e=-1+lr$ mod $q.$ Next
\begin{align*}
a^{2}h  & =\phi^{2}(h)=\phi(\left(  -1+lr\right)  h)\\
& =\phi(-h)\phi(lrh)=-eh+lrh\\
& =(1-lr+lr)h=h,
\end{align*}
so it follows that $a^{2}=1$ $\operatorname{mod}$ $q$ or that $q|(a-1)(a+1).$
From Table 3, for $r\geq8,$ there are four possibilities for $a,$ namely
$a=1,q/2-1,q/2+1,q-1$ $\operatorname{mod}$ $q.$ The case $a=1$ is eliminated
unless $\left\vert \overline{S}\right\vert =1,2.$ In the case $a=q/2-1,$
$S_{C}$ is contained in the subgroup $\left\langle \frac{q}{2}h\right\rangle $
of order $2.$ If $S_{C}=\left\langle \frac{q}{2}h\right\rangle $ then
$\phi(x)=-x+\frac{q}{2}x,$ where $\frac{q}{2}x\in S_{C}$ so that
$\phi(\overline{x})=-\overline{x}.$ In the case $e=q/2+1,$ $S_{C}$ is
contained in subgroup $\left\langle 2h\right\rangle $ of order $q/2.$ If
$S_{C}=\left\langle 2h\right\rangle $ then equation \ref{eq-phi3} holds
trivially otherwise equation \ref{eq-phi3} fails. The case $a=r-1=-1$
$\operatorname{mod}$ $r$ is eliminated unless $\left\vert S_{C}\right\vert
=1,2.$ Finally if $q=4$ then the cases $e=1,3$ both lead to the solution
$S_{C}\backsimeq\overline{S}=C_{2}.$

Our remaining examples come from various lines of Table 4 where $\Gamma
_{N}=T(2,d,2d)$ or a variant. The computer calculations in \cite{Brou-Woo2}
show that the only possible overgroup $\Gamma_{A}$ of $\Gamma_{N}=T(2,d,2d) $
for generic $d$ is $\Gamma_{A}=T(2,3,2d)$ with monodromy group $M(\Gamma
_{A},\Gamma_{N})=\Sigma_{3}.$

\begin{example}
Suppose that $\Gamma_{N}=T(2,2d,d)$ and that $K=C_{k},$ $k>2$ as in Table 4 We
first determine the possible $N$ and then show that a full automorphism group
$A$ with $M(A,N)=\Sigma_{3}$ as described in Example \ref{ex-parametric} is
not possible. From Table 4 line 5 we see that $d$ is odd $\left\vert
C\right\vert =2\frac{d}{k}$ that $N$ has an element order $2d=\left\vert
K\right\vert \left\vert C\right\vert =\left\vert N\right\vert ,\ $thus $N$ is
cyclic. By direct construction there are cyclic actions with signature
$(d,2d,2)$ since $d$ is odd. Since $M(A,N)=\Sigma_{3}$ then $\left\vert
\mathrm{Core}_{A}(N)\right\vert =d$ and is a proper cyclic subgroup of $C.$
But $\mathrm{Core}_{A}(N)$ is a proper subgroup of $C$ of order $d$ and hence
$\mathrm{Core}_{A}(N)=\mathrm{Core}_{A}(C)$ contradicting weak normality. Even
dropping the assumption that $C$ is weakly malnormal, does not yield any
larger groups at least in the case when we assume that $d$ is relatively prime
to $6$.
\end{example}

\begin{example}
Suppose that $\Gamma_{N}=T(2,2d,d)$ and that $K=D_{k},$ $k>2$ as given in line
9 of Table 4. We first determine the possible $N$ and then determine whether
there can be a full automorphism group $A$ with $M(A,N)=\Sigma_{3}$ as
described in Example \ref{ex-parametric}. From Table 4 we see that $n=d=ek $
and $\left\vert N\right\vert =2ek^{2}.$ For simplicity's sake let us assume
that $k$ and $e$ are coprime odd numbers, coprime to $\left\vert \Sigma
_{3}\right\vert $.

Let us determine the various $K$-actions and $K$-fixed points on $X$ the set
of generating vectors of $C$ defined in equation \ref{eq-defofX}. The
generating vector for $C$ has two elements of order $e$ in one $K$-orbit
$\mathcal{O}_{1}$ and $k$ elements of order $n$ in another $K$-orbit
$\mathcal{O}_{2}$. For the orbit $\mathcal{O}_{1}$ the point stabilizers are
$C_{k}$ so no condition is imposed and for points of $\mathcal{O}_{2}$ the
$z_{i}$ must be $\phi$-invariant. As the $z_{i}$ are generators $\phi$ must be
trivial, and hence $D_{k}$ acts trivially on $C$. A typical $K$-invariant
vector has the form $(g,g,h,\ldots,h)$, where $g$ has order $e$ and each of
the $k$ repeats of $h$ has order $n$. Thus $h$ can be any of the $\phi(n)$
elements of order $n.$ We must have $g^{2}h^{k}=1$ or $g^{2}=h^{-k}.$ Since
$h^{-k}$ has exact order $e$ then $g^{2}=h^{-k}$ has a unique solution in
$C_{e}$ since squaring is injective. Now let us find the possible groups $N$.
According to our previous analysis the subgroup $E\subseteq N$ is isomorphic
to $C\times C_{k}.$ Since $2$ divides $\left\vert N\right\vert $ there is an
element $y$ $\in$ $N-E$ of order $2.$ Putting everything together $N$
$\backsimeq$ $C_{e}\times C_{2}\ltimes(C_{k}\times C_{k}),$ with $y$ acting by
$(z_{1},z_{2})\rightarrow(z_{1},-z_{2})$ on $C_{k}\times C_{k}.$ The subgroup
$C_{2}\ltimes\{(0,z):z\in C_{k}\}$ is a subgroup of $N$ mapping onto $K$. If
we let $C_{e}=\left\langle x\right\rangle ,$ $C_{2}=\left\langle
y\right\rangle ,C_{k}\times C_{k}=\left\langle z_{1}\right\rangle
\times\left\langle z_{2}\right\rangle $ then a generating $(2,2d,d)$-vector is
$(y,yxz_{1}z_{2},\left(  xz_{1}z_{2}\right)  ^{-1}).$ For, $yxz_{1}%
z_{2}yxz_{1}z_{2}=x^{2}z_{1}^{2}$ has order $n,$ and $xz_{1}z_{2}$ and
$x^{2}z_{1}^{2}$ generate $C_{e}\times(C_{k}\times C_{k}),$ so the vector
generates all of $N.$

Now let us find $A.\ $Since $M(A,N)=\Sigma_{3},$ then $\mathrm{Core}_{A}(N)$
is of index 2 in $N,$ and so $\mathrm{Core}_{A}(N)=E=C_{e}\times C_{k}\times
C_{k}.$ The subgroups $C_{e}$ and $C_{k}\times C_{k}$ are characteristic
subgroups of $E$ and hence they are normal in $A.$ The action of $\Sigma_{3} $
on $C_{e}$ is induced by the action of $\omega(g_{r})$ and its conjugates, and
so $A$ acts trivially on $C_{e}.$ Using Sylow theory and the fact that
$k^{2}e$ and $6$ are coprime we can show that $E\rightarrow A\rightarrow$
$\Sigma_{3}$ is split and $A$ $\backsimeq$ $C_{e}\times\Sigma_{3}\ltimes
(C_{k}\times C_{k}).$ But the $A$-action cannot have a generating
$(2,3,2d)$-vector because the image of the vector under the epimorphism
$A\rightarrow C_{e}$ can only have signature $(1,1,e)$, or $(1,1,1)$ a
contradiction in both cases. Therefore, $N$ is the full automorphism group. In
the Fermat case $e=1$ and the contradiction is eliminated.
\end{example}

\begin{example}
Our last example will have a non-trivial action of $K$ on $C.$ Suppose that
$\Gamma_{N}=T(2,d,2d)$ and that $K=D_{k},$ $k>2$ as given in line 13 of Table
4. From Table 4 we see that $d$ is even and $2d=ek.$ There are many cases to
consider depending on the parity of $e$ and $k.$ We shall assume that
$k=4l,$where $e$ and $l$ are odd coprime integers. The branching data of $C$
collected into $K$-orbits is $(1^{k},\left(  el\right)  ^{k},e^{2}).$ Upon
permutation and dropping trivial generators we get a signature of
$(e,e,\left(  el\right)  ^{k}),$ so $n=el,$ and $C=C_{e}\times C_{l}.$ Let
$(a,b,c)$ be a generating $(2,2,k)$-vector of $K,$ with $a$ the stabilizer of
a $1^{k}$ orbit, $b$ the stabilizer of an $\left(  el\right)  ^{k}$ orbit and
$\ c$ the stabilizer of an $e^{2}$ orbit. Then there is no restriction on $a,$
$b$ must fix an element of order $n=el$ and $c$ must fix an element of order
$e.$ A non-trivial action of $K$ on $C=C_{e}\times C_{l}$ satisfying the fixed
point restriction is
\begin{align*}
a  & :(x,y)\rightarrow(x,y^{-1}),\\
b  & :(x,y)\rightarrow(x,y),\\
c  & :(x,y)\rightarrow(x,y^{-1}).
\end{align*}
Using Sylow subgroup analysis and previous techniques one can show that
$N\backsimeq D_{k}\ltimes C$ with action defined above. Also as previously
argued the $N$-action does not extend to a $(2,3,2d)$-action.
\end{example}

\end{document}